\documentclass[11pt]{article}

\RequirePackage[amsthm,amsmath]{imsart}

\usepackage[mathscr]{eucal}
\usepackage{graphicx}
\usepackage{latexsym}
\usepackage{amssymb}
\usepackage{psfrag}
\usepackage{epsfig}
\usepackage{enumerate}
\DeclareMathAlphabet{\mathpzc}{OT1}{pzc}{m}{it}

\include{psbox}

\usepackage{amsfonts}
\usepackage{graphicx}
\usepackage{amsfonts}
\usepackage{color}
\usepackage[usenames,dvipsnames]{xcolor}
\usepackage[textsize=small]{todonotes}

\allowdisplaybreaks[1]

\usepackage[numbers]{natbib}

\begin{document}


\newtheorem{proposition}{Proposition}[section]
\newtheorem{theorem}[proposition]{Theorem}
\newtheorem{corollary}[proposition]{Corollary}
\newtheorem{lemma}[proposition]{Lemma}
\newtheorem{conjecture}[proposition]{Conjecture}
\newtheorem{question}[proposition]{Question}
\newtheorem{definition}[proposition]{Definition}
\newtheorem{algorithm}[proposition]{Algorithm}
\newtheorem{assumption}[proposition]{Assumption}
\newtheorem{condition}[proposition]{Condition}
\renewcommand{\thefootnote}{\color{red}\arabic{footnote}}
\newcommand{\esquare}{\begin{flushright}$\square$\end{flushright}}
\numberwithin{equation}{section}
\numberwithin{proposition}{section}
\renewcommand{\theenumi}{\roman{enumi}}


%
%
\newcommand{\skp}{\vspace{\baselineskip}}
\newcommand{\noi}{\noindent}
\newcommand{\osc}{\mbox{osc}}
\newcommand{\lfl}{\lfloor}
\newcommand{\rfl}{\rfloor}

\theoremstyle{remark}
\newtheorem{example}{\bf Example}[section]
\newtheorem{remark}{\bf Remark}[section]

\newcommand{\img}{\imath}
\newcommand{\iy}{\infty}
\newcommand{\eps}{\epsilon}
\newcommand{\veps}{\varepsilon}
\newcommand{\del}{\delta}
\newcommand{\Rk}{\mathbb{R}^k}
\newcommand{\RR}{\mathbb{R}}
\newcommand{\spa}{\vspace{.2in}}
\newcommand{\V}{\mathcal{V}}
\newcommand{\E}{\mathbb{E}}
\newcommand{\I}{\mathbb{I}}
\newcommand{\PP}{\mathbb{P}}
\newcommand{\sgn}{\mbox{sgn}}
\newcommand{\ti}{\tilde}

\newcommand{\QQ}{\mathbb{Q}}

\newcommand{\XX}{\mathbb{X}}
\newcommand{\XXz}{\mathbb{X}^0}

\newcommand{\lan}{\langle}
\newcommand{\ran}{\rangle}
\newcommand{\lf}{\lfloor}
\newcommand{\rf}{\rfloor}
\def\wh{\widehat}
\newcommand{\defn}{\stackrel{def}{=}}
\newcommand{\txb}{\tau^{\epsilon,x}_{B^c}}
\newcommand{\tyb}{\tau^{\epsilon,y}_{B^c}}
\newcommand{\tilxb}{\tilde{\tau}^\eps_1}
\newcommand{\pxeps}{\mathbb{P}_x^{\eps}}
\newcommand{\non}{\nonumber}
\newcommand{\dist}{\mbox{dist}}

\newcommand{\Om}{\mathnormal{\Omega}}
\newcommand{\om}{\omega}
\newcommand{\vph}{\varphi}
\newcommand{\Del}{\mathnormal{\Delta}}
\newcommand{\Gam}{\mathnormal{\Gamma}}
\newcommand{\Sig}{\mathnormal{\Sigma}}

\newcommand{\tilyb}{\tilde{\tau}^\eps_2}
\newcommand{\beq}{\begin{eqnarray*}}
\newcommand{\eeq}{\end{eqnarray*}}
\newcommand{\beqn}{\begin{eqnarray}}
\newcommand{\eeqn}{\end{eqnarray}}
\newcommand{\ink}{\rule{.5\baselineskip}{.55\baselineskip}}

\newcommand{\bt}{\begin{theorem}}
\newcommand{\et}{\end{theorem}}
\newcommand{\deps}{\Del_{\eps}}
\newcommand{\dbl}{\mathbf{d}_{\tiny{\mbox{BL}}}}

\newcommand{\be}{\begin{equation}}
\newcommand{\ee}{\end{equation}}
\newcommand{\bes}{\begin{equation*}}
\newcommand{\ees}{\end{equation*}}
\newcommand{\ba}{\begin{aligned}}
\newcommand{\ea}{\end{aligned}}
\newcommand{\ac}{\mbox{AC}}
\newcommand{\BB}{\mathbb{B}}
\newcommand{\VV}{\mathbb{V}}
\newcommand{\DD}{\mathbb{D}}
\newcommand{\KK}{\mathbb{K}}
\newcommand{\HH}{\mathbb{H}}
\newcommand{\TT}{\mathbb{T}}
\newcommand{\CC}{\mathbb{C}}
\newcommand{\ZZ}{\mathbb{Z}}
\newcommand{\SSS}{\mathbb{S}}
\newcommand{\EE}{\mathbb{E}}
\newcommand{\NN}{\mathbb{N}}

\newcommand{\clg}{\mathcal{G}}
\newcommand{\clb}{\mathcal{B}}
\newcommand{\cls}{\mathcal{S}}
\newcommand{\clc}{\mathcal{C}}
\newcommand{\clj}{\mathcal{J}}
\newcommand{\clm}{\mathcal{M}}
\newcommand{\clx}{\mathcal{X}}
\newcommand{\cld}{\mathcal{D}}
\newcommand{\cle}{\mathcal{E}}
\newcommand{\clv}{\mathcal{V}}
\newcommand{\clu}{\mathcal{U}}
\newcommand{\clr}{\mathcal{R}}
\newcommand{\clt}{\mathcal{T}}
\newcommand{\cll}{\mathcal{L}}
\newcommand{\clz}{\mathcal{Z}}
\newcommand{\clq}{\mathcal{Q}}
\newcommand{\clo}{\mathcal{O}}

\newcommand{\cli}{\mathcal{I}}
\newcommand{\clp}{\mathcal{P}}
\newcommand{\cla}{\mathcal{A}}
\newcommand{\clf}{\mathcal{F}}
\newcommand{\clh}{\mathcal{H}}
\newcommand{\N}{\mathbb{N}}
\newcommand{\Q}{\mathbb{Q}}
\newcommand{\bfx}{{\boldsymbol{x}}}
\newcommand{\bfh}{{\boldsymbol{h}}}
\newcommand{\bfs}{{\boldsymbol{s}}}
\newcommand{\bfm}{{\boldsymbol{m}}}
\newcommand{\bff}{{\boldsymbol{f}}}
\newcommand{\bfb}{{\boldsymbol{b}}}
\newcommand{\bfw}{{\boldsymbol{w}}}
\newcommand{\bfz}{{\boldsymbol{z}}}
\newcommand{\bfu}{{\boldsymbol{u}}}
\newcommand{\bfell}{{\boldsymbol{\ell}}}
\newcommand{\bfn}{{\boldsymbol{n}}}
\newcommand{\bfd}{{\boldsymbol{d}}}
\newcommand{\bfbeta}{{\boldsymbol{\beta}}}
\newcommand{\bfzeta}{{\boldsymbol{\zeta}}}
\newcommand{\bfnu}{{\boldsymbol{\nu}}}

\newcommand{\bfl}{{ L}}
\newcommand{\bfc}{{C}}
\newcommand{\bfg}{{\bf G}}
\newcommand{\bfi}{{1}}
\newcommand{\bfk}{{\bf k}}
\newcommand{\bfDD}{{D}}

\newcommand{\bft}{{\bf T}}

\newcommand{\hatq}{\hat{Q}}
\newcommand{\hata}{\hat{A}}
\newcommand{\hatd}{\hat{D}}
\newcommand{\hats}{\hat{S}}
\newcommand{\hatx}{\hat{X}}
\newcommand{\haty}{\hat{Y}}
\newcommand{\hati}{\hat{I}}
\newcommand{\hatw}{\hat{W}}
\newcommand{\bard}{\bar{D}}
\newcommand{\barq}{\bar{Q}}
\newcommand{\bara}{\bar{A}}
\newcommand{\bars}{\bar{S}}
\newcommand{\bari}{\bar{I}}

\newcommand{\bart}{\bar{T}}
\newcommand{\Go}{\Rightarrow}

\newcommand{\curvz}{{\bf \mathpzc{z}}}
\newcommand{\curvx}{{\bf \mathpzc{x}}}
\newcommand{\curvi}{{\bf \mathpzc{i}}}
\newcommand{\curvs}{{\bf \mathpzc{s}}}
\newcommand{\blip}{\mathbb{B}_1}

\newcommand{\BM}{\mbox{BM}}

\newcommand{\tac}{\mbox{\scriptsize{AC}}}

\newcommand{\bfA}{{\boldsymbol{A}}}
\newcommand{\bfB}{{\boldsymbol{B}}}
\newcommand{\bfC}{{\boldsymbol{C}}}
\newcommand{\bfD}{D}
\newcommand{\bfG}{{\boldsymbol{G}}}
\newcommand{\bfH}{{\boldsymbol{H}}}
\newcommand{\vt}{\Theta}
\newcommand{\dmnk}{\mathcal{D}^{n,k}_m}
\newcommand{\hmnk}{\mathcal{H}^{n,k}_m}

\newcommand{\bi}{\begin{itemize}}
\newcommand{\ei}{\end{itemize}}
\newcommand{\var}{\mbox{Var}}
\newcommand{\barn}{\bar{N}}
\newcommand{\barg}{\bar{G}}
\newcommand{\barr}{\bar{R}}
\newcommand{\hatn}{\hat{N}}
\newcommand{\hatg}{\hat{G}}
\newcommand{\hatr}{\hat{R}}
\newcommand{\D}{{\rm d}}
\newcommand{\hatm}{\hat{M}}
%


\begin{frontmatter}
\title{Diffusion Approximations for Double-ended Queues with Reneging in Heavy Traffic}

 \runtitle{Double-ended queues with reneging}

\begin{aug}
\author{Xin Liu\\ \ \\
}
\end{aug}

\today

\skp

\begin{abstract}
\noi
We study a double-ended queue which consists of two classes of customers. Whenever there is a pair of customers from both classes, they are matched and leave the system immediately. The matching follows first-come-first-serve principle. If a customer from one class cannot be matched immediately, he/she will stay in a queue and wait for the upcoming arrivals from the other class. Thus there cannot be non-zero numbers of customers from both classes simultaneously in the system. We also assume that each customer can leave the queue without being matched because of impatience. The arrival processes are assumed to be independent renewal processes, and the patience times for both classes are generally distributed. Under suitable heavy traffic conditions, assuming that the diffusion-scaled queue length process is stochastically bounded, we establish a simple asymptotic relationship between the diffusion-scaled queue length process and the diffusion-scaled offered waiting time processes, and further show that the diffusion-scaled queue length process converges weakly to a diffusion process. We also provide a sufficient condition for the stochastic boundedness of the diffusion-scaled queue length process. At last, the explicit form of the stationary distribution of the limit diffusion process is derived. 
\\

\noi {\bf AMS 2010 subject classifications:} Primary: 60F05, 60K25, 90B22; Secondary: 60K05. 

\noi {\bf Keywords:} Double-ended queues; Matching systems; First-Come-First-Serve; Customer abandonment; Generally distributed patience times; Hazard rate scaling; Diffusion approximations; Stationary distribution; Heavy traffic.
\end{abstract}

\end{frontmatter}

%
%
%
%
%
%
%
%
%
%

\section{Introduction}

Consider a simple matching system which consists of two classes of customers. Whenever there is a pair of customers from both classes, they are matched and leave the system immediately. The matching follows first-come-first-serve principle. If a customer from one class cannot be matched immediately, he/she will stay in a queue and wait for the upcoming arrivals from the other class. Thus there cannot be non-zero numbers of customers from both classes simultaneously in the system. Customers are assumed to be impatient and they can leave the system without being matched. Such system forms a double-ended queueing system, which is schematically shown in Figure \ref{fig1}. We assume that the arrival processes are independent renewal processes, and the patience times for customers of each class are IID with a general distribution. Under certain conditions, we establish a simple asymptotic relationship between the diffusion-scaled queue length process and the diffusion-scaled offered waiting time processes, and show that the diffusion-scaled queue length process converges weakly to a diffusion process. Those conditions consist of a suitable heavy traffic condition, a mild condition on patience-time distributions, and an assumption that the diffusion-scaled queue length process is stochastically bounded. We further show that under the heavy traffic condition and a proper scaling of the hazard rate functions of the patience-time distributions, the diffusion-scaled queue length is stochastically bounded. Finally, the unique stationary distribution of the diffusion limit is derived.  Diffusion approximations for double-ended queues with renewal arrivals and exponential patience times are studied in \cite{liu2015}, and our current work extends the heavy traffic diffusion approximation result in \cite{liu2015} to a setting with generally distributed patience times. 

Double-ended queues have been studied for many applications, including taxi-service systems, financial markets, assembly systems, perishable inventory systems, and organ transplant systems. Double-ended queues with Poisson arrivals were first introduced by Kendall in \cite{kendall51}, and Dobbie \cite{dobbie61} studied the probability of a given queue size of double-ended queues proposed by Kendall. Later on, Giveen \cite{Giveen63} considered double-ended queues with time-dependent Poisson arrivals, and studied the limiting distributions, and Kashyap \cite{kashyap1966double} studied a taxi service system with limited waiting space, where the arrivals of taxies and passengers are assumed to be Poisson processes, and derived the steady state distribution. In \cite{conolly2002double}, Conolly et al. studied the impatience behavior of customers under the assumption of exponential arrivals, services, and patience times. Applications of double-ended queues on financial markets, networks with synchronization nodes, perishable inventory systems, and organ transplant systems can be found in \cite{adm14, zenios, prabhakar2000synchronization, perry1999perishable, bdps}, etc. When renewal arrivals and/or generally distributed patience times are considered, the explicit form of the limiting distribution becomes intractable. Degirmenci \cite{talayasymptotic} studied the asymptotic behaviors of the Markovian double-ended queue, and conjectured that the result holds for the setting with renewal arrivals and general patience times. In \cite{kim2010simulation}, Kim et al. studied the extended double-ended queue which allowed bulk arrivals, general patience times, positive processing times, etc, using simulation methods. Zenios \cite{zenios} applied Laplace's method to a double-ended queue where the inter-arrivals of supply were generally distributed, and demand arrived according a Poisson process, and only demand had abandonments with generally distributed patience times. In this work, we develop rigorous diffusion approximations for double-ended queues with renewal arrivals and generally distributed patience times under appropriate asymptotic regimes. Diffusion approximations of (one-sided) queueing systems with reneging have been studied in both conventional heavy traffic regime and Halfin-Whitt regime. Roughly speaking, in both regimes, the queueing system is nearly critically loaded. The conventional heavy traffic analysis can be found in \cite{WardGlynn03, WardGlynn05, Reed08, Lee11}, while \cite{Garnet02, Man05, Dai10, Man12, DaiHe, ReedTezcan12} study many-server queues in Halfin-Whitt regime.

In this work, our main assumption on patience time distributions is similar to that considered by Lee and Weerasinghe \cite{Lee11}, which requires the patience time distribution functions satisfy some scaling limit results, and are more general than the commonly used assumptions in the above literatures  (see Assumption \ref{patience} and Remark \ref{remark} (iii)). Our main results, Theorems \ref{relation-thm} and \ref{diffusion-thm}, are established under this assumption, the heavy traffic assumption (Assumption \ref{htc}), and the assumption that the diffusion-scaled queue length process is stochastically bounded. We then consider the hazard rate scaling as in \cite{Reed08, ReedTezcan12, kang10} (see Assumption \ref{hazard}), and show  in Corollary \ref{cor} that under such hazard rate scaling and the heavy traffic condition, the diffusion-scaled queue length process is stochastically bounded and so all the results hold. It is worth noting that the condition that the diffusion-scaled queue length process is stochastically bounded is considered by Dai and He \cite{DaiHe}, and they establish the diffusion approximations for a $G/G/n+GI$ queue under this condition, and a more strict condition on patience time distribution in Halfin-Whitt regime. Furthermore, Reed and Tezcan \cite{ReedTezcan12} obtain a heavy traffic limit for $GI/M/n + GI$ queue by assuming that the patience time distribution satisfies an appropriate hazard rate scaling.  

The rest of the paper is organized as follows. In Section \ref{sec:MF}, we introduce the model of the double-ended queue with renewal arrivals and generally distributed patience times. In Section \ref{sec:AF}, we present the asymptotic framework by introducing the main assumptions.  All the main results are collected in Section \ref{sec:results}, and their proofs are provided in Section \ref{proofs}, and Appendix.

We use the following notation. Denote by $\mathbb{R},$ $\mathbb{R}_+,$ $\mathbb{Z}$, and $\mathbb{N}$ the sets of real numbers, nonnegative real numbers, integers, and positive integers, respectively. For a real number $a$, define $a^+ = \max\{a, 0\}$ and $a^- = \max\{0,-a\}.$ Similarly, for a real function $f$ defined on $[0,\infty)$, define $f^+(t) = \max\{0, f(t)\}$ and $f^-(t) = \max\{0, -f(t)\}, \ t\ge 0.$ Denote by $\mathcal{D}([0,\infty); \RR)$ the space of right continuous functions with left limits defined from $[0,\infty)$ to $\RR$ with the usual Skorohod topology. For $x\in \mathcal{D}([0,\infty); \RR)$, let 
$$\|x\|_t = \sup_{s\in[0, t]} |x(s)|, \ t\ge0, $$ and 
$$\|x\|_{s,t} = \sup_{u\in[s, t]} |x(u)|, \ t\ge s\ge0.$$
 A function $f: [0,\infty)\to [0,\infty)$ is called locally Lipschitz continuous if for any $t\in [0,\infty),$ there exists $\kappa\in (0,\infty)$ (may depend on $t$) such that for $x_1, x_2 \in [0, t]$,
\bes
|f(x_1) - f(x_2)| \le \kappa |x_1 - x_2|.
\ees
For a stochastic process $X$, we will use the notation $X(t)$ and $X_t$ interchangeably. For a semimartingale $Y\in \mathcal{D}([0,\infty); \RR)$, we denote by $[Y]$ the quadratic variation of $Y$. A sequence of stochastic processes $\{X^n\}$ is called stochastically bounded if for each $T\ge 0$, 
\[
\lim_{K\to\infty}\limsup_{n\to\infty} \PP(\|X^n\|_T > K) =0.
\]

\section{Network model}\label{sec:MF}

We study a double-ended queue which consists of two classes of customers -- Classes $1$ and $-1$. Whenever there is a pair of customers from both classes, they are matched and leave the system immediately. The matching follows first-come-first-served principle. If a customer from one class cannot be matched immediately, he/she will stay in a queue and wait for the upcoming arrivals from the other class. We also assume that each customer is impatient. A double-ended queue is schematically given in
Figure \ref{fig1}.
\begin{center}
\begin{figure}[h]
\input{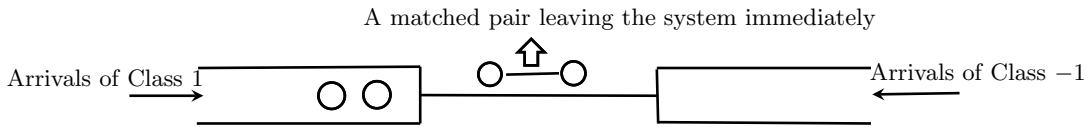}
  \caption{A double-ended Queue.} \label{fig1}
\end{figure}
\end{center}
Let $Q(t)$ be the length of the double-ended queue at time $t$. Different from queue length processes in one-sided queueing systems, here $Q(t) \in \mathbb{Z}$, and we let $Q^+(t)$ denote the number of customers of Class $1$ waiting in the queue at time $t$, and $Q^-(t)$ the number of customers of Class $-1$ waiting in the queue at time $t$. For $t\ge 0$, let $$Q_1(t) = Q^{+}(t) \  \ \mbox{and} \ \ Q_{-1}(t) = Q^{-}(t).$$ 
Without loss of generality, we assume $Q_1(0)\ge 0$ and $Q_{-1}(0)=0$. 

Let $(\Omega, \mathcal{F}, \PP)$ be a complete probability space satisfying the usual conditions. All the random variables and stochastic processes are assumed to be defined on this space. Let $i=\pm1$. We assume the inter-arrival times
for customers of Class $i$ are i.i.d. random variables $\{u_{i,k}: k\in \mathbb{N}\}$, and $u_{i,1}$ has mean $\frac{1}{\lambda_i}$ and standard deviation $\sigma_i$.
Define
\begin{equation*}
\begin{aligned}
N_i(t) &= \max\left\{k\in\NN: \sum_{l=1}^k u_{i,l}\leq t\right\}.
\end{aligned}
\end{equation*}
The renewal process  $N_i$ is the arrival process for Class $i$. We assume that the patience times of customers of Class $i$ are given by an i.i.d. sequence $\{d_{i,k}: k\in\mathbb{Z}\}.$ For $k\in\NN$, $d_{i,k}$ represents the patience time of the $k$-th customer of Class $i$ who enters the system after time $0$, and for $k\in -\NN \cup\{0\}$, $d_{i,k}$ is the patience time of the $(-k+1)$-st customer of Class $i$ who enters the system prior to time $0$ (if such customers exist). We assume that $d_{i,1}$ has cumulative distribution function $F_i$. Finally, we assume $Q(0)$, $\{u_{1,k}: k\in \mathbb{N}\}, \{u_{-1,k}: k\in\mathbb{N}\},$ $\{d_{1,k}: k\in\mathbb{Z}\},$ and $\{d_{-1,k}: k\in\mathbb{Z}\}$ are independent. 

For $k\in \ZZ$, define $t_{i,k}$ to be the arrival time of the $k^{th}$ customer of Class $i$. More precisely, for $k\in\NN$, $t_{i,k} = \sum_{l=1}^k u_{i,l}$, and for $k\in -\NN\cup\{0\}$, we let $t_{i,k}=0$, that is all customers, who arrive prior to time $0$, have arrival times equal to $0$. For $k\in\mathbb{Z},$ let $w_{i,k}$, which is called the {\it offered waiting time}, denote the waiting time that the $(k1_{\{k>0\}}+(-k+1)1_{\{k\le 0\}})$-th customer of Class $i$ needs to experience assuming her/his patience time is infinite. For a customer who arrives prior to time $0$, we assume her/his waiting time starts to count at time $0.$
Let $G_i(t)$ denote the number of Class $i$ customers who abandon the system by time $t$. Then for $t\ge 0$, 
\begin{equation}
G_i(t) = \sum_{k=-Q_i(0)+1}^{N_i(t)} 1_{\{d_{i,k} \le w_{i,k}, \ t_{i,k}+d_{i,k}\le t\}}.
\end{equation}
The process $G_i$ will be called the {\it abandonment process} for Class $i$. The {\em queue length process} now can be formulated as follows: 
\begin{equation}\label{queue-length}
Q(t) =  Q(0) + N_1(t) - N_{-1}(t) - G_1(t) + G_{-1}(t).
\end{equation}

To study the offered waiting times, we next define $R_i(t)$ to be the number of Class $i$ customers who have arrived by time $t$, and will abandon the system eventually. Then for $t\ge 0,$
\begin{equation}
R_i(t)  = \sum_{k=-Q_i(0)+1}^{N_i(t)} 1_{\{d_{i,k} \le w_{i,k}\}}.
\end{equation}
For Class $1$, there are $Q_1(0) = Q^+(0)$ customers entering the system prior to time $0$, and for such customers, define
\bes
R_{1,k}(0) = \sum_{l=k+1}^0 1_{\{d_{1,l} \le w_{1,l}\}}, \ k= -Q_1(0)+1, \ldots, -1,  \ \mbox{and} \  R_{1,0} =0.
\ees
Here $R_{1,k}, k =-Q_1(0)+1, \ldots, -1, 0,$ represents the number of the first $-k$ customers of Class $1$ who arrives prior to time $0$, and will abandon the system eventually. In the following, we characterize the offered waiting times of customers in terms of $\{t_{-i,k}, k\in\ZZ\}$, $\{R_i(t), t\ge 0\}$ and $\{R_{1,k}(0), k=-Q_1(0)+1, \ldots, -1, 0\}.$ First, it is clear that the offered waiting time of the first customer of Class $1$ who arrives prior to time $0$ is $t_{-1,1}$ (i.e., the arrival time of the first customer of Class $-1$). For $k=-Q_1(0)+1, \ldots, -1$, assuming he/she is patient, the $(-k+1)$-st customer of Class $1$ arriving before time $0$ will be matched with the $(-k+1-R_{1,k}(0))$-th customer of Class $-1$. Thus the offered waiting time of the $(-k+1)$-st customer of Class $1$ who arrives before time $0$ is 
\begin{equation}\label{waiting-negative}
w_{1,k} = t_{-1, -k+1 -R_{1,k}(0)}. 
\end{equation}
We next consider the customers arriving after time $0$. For $k\in \NN$, at time $t_{1,k}$, the $k$-th customer of Class $1$ arrives at the system. If $Q_{-1}(t_{1,k}-)>0$ (i.e., there are customers of Class $-1$ waiting at time $t_{1,k}$), then the offered waiting time is
\begin{equation}\label{eq:23}
w_{1,k}=0.
\end{equation} 
If $Q_{-1}(t_{1,k}-)=0$ (i.e., there is no customer of Class $-1$ waiting at time $t_{1,k}$), this $k$-th customer of Class $1$ must wait in the system for the $(k+Q_1(0)- R_1(t_{1,k}-)+ G_{-1}(t_{1,k}-))$-th customer of Class $-1$, and the offered waiting time is
\bes
w_{1,k} = t_{-1, k+Q_1(0)-R_1(t_{1,k}-)+G_{-1}(t_{1,k}-)} - t_{1,k}.
\ees
Next noting that $R_{-1}(t)=G_{-1}(t)$ when $Q_{-1}(t) =0$, we have 
\begin{equation}\label{eq:21}
w_{1,k} = t_{-1, k+Q_1(0)-R_1(t_{1,k}-)+R_{-1}(t_{1,k}-)} - t_{1,k}.
\end{equation}
  Combining \eqref{waiting-negative}, \eqref{eq:23} and \eqref{eq:21}, we have that
\begin{equation}\label{owt-1}
w_{1,k} = \begin{cases}
 \left(  t_{-1, k+Q_1(0)-R_1(t_{1,k}-)+R_{-1}(t_{1,k}-)} - t_{1,k}\right) 1_{\{Q_{-1}(t_{1,k}-)=0\}}, & k\in\NN, \\
t_{-1, -k+1 -R_{1,k}(0)}, & -Q_1(0)+1 \le k\le 0.
\end{cases}
\end{equation}
Using similar arguments, we have for Class $-1$, 
\begin{equation}\label{owt-2}
w_{-1,k} = \left(  t_{1, k-R_{-1}(t_{-1,k}-)-Q_1(0)+R_{1}(t_{-1,k}-)} - t_{-1,k}\right) 1_{\{Q_{1}(t_{-1,k}-)=0\}}, \ k\in\NN.
\end{equation}
It is shown in Lemma \ref{reformulate} that for $i=\pm 1$ and $k\in\NN$,
\begin{align}\label{owt}
w_{i,k} = \left[t_{-i, k+Q_i(0)-R_i(t_{i,k}-)-Q_{-i}(0)+R_{-i}(t_{i,k}-)} - t_{i,k}\right]^+.
\end{align}
We thus define the virtual waiting time process $W_i(t)$, which can be interpreted as the amount of time that a hypothetical customer of Class $i$ arriving right after time $t$ will have to wait in the system, as follows. For $t\ge 0,$
\begin{equation}\label{vwt}
W_i(t) = \left[ t_{-i,N_i(t) + 1 + Q_i(0) - R_i(t) - Q_{-i}(0)+ R_{-i}(t)} - t\right]^+.
\end{equation}
It is clear that $W_i(t_{i,k}-)=w_{i,k}$ for $k\in \mathbb{N}.$

\begin{remark}\label{remark_1}\hfill
\begin{itemize}
\item[\rm (i)] In \eqref{owt-2}, it is possible that the subscript $k-R_{-1}(t_{-1,k}-)-Q_1(0)+R_{1}(t_{-1,k}-) \le 0$. When this happens, we have $Q_{1}(t_{-1,k}-)>0$, and $w_{-1,k}=0$. (See Lemma 5.1 for its proof.) In fact, this $k$-th customer of Class $-1$ will be matched with a customer of Class $1$ who arrives prior to time $0$ and are waiting in the queue at time $t_{-1,k}$. 
%
\item[\rm (ii)] Lemma \ref{reformulate} shows that for $k\in\NN$,  in \eqref{owt-1}, when $Q_{-1}(t_{1,k}-)=0$, we have 
\[ 
t_{-1, k+Q_1(0)-R_1(t_{1,k}-)+R_{-1}(t_{1,k}-)} \ge t_{1,k},
\] 
and in \eqref{owt-2}, when $k-R_{-1}(t_{-1,k}-)-Q_1(0)+R_{1}(t_{-1,k}-) > 0$, and $Q_{1}(t_{-1,k}-)=0$, we have
\[
t_{1, k-R_{-1}(t_{-1,k}-)-Q_1(0)+R_{1}(t_{-1,k}-)} \ge t_{-1,k}.
\]

\end{itemize}
\end{remark}

\section{Asymptotic framework}\label{sec:AF}

We consider a sequence of double-ended queues indexed by $n\in\mathbb{N}$. For the $n$-th system, all the notation introduced in Section \ref{sec:MF} is carried forward except that we append a superscript $n$ to all quantities to indicate the dependence of parameters, random variables, and stochastic processes on $n$. Let  $i=\pm 1$. In particular, we assume all the random variables $\{u_{i,k}^n: k\in \mathbb{N}\}$, $\{d_{i,k}^n: k\in\mathbb{Z}\}$, $\{t_{i,k}^n: k\in\mathbb{Z}\}, \{w_{i,k}^n: k\in\mathbb{Z}\}$, and stochastic processes $Q^n, N^n_i, G^n_i, R^n_i, W^n_i$ are defined on the space $(\Omega^n, \mathcal{F}^n, \PP^n)$. The expectation operator with respect to $\PP^n$ will be denoted by $\E^n$. Finally, the cumulative distribution function of $d^n_{i,1}$ is $F^n_i$, and the mean and standard deviation of $u^n_{i,1}$ are $\frac{1}{\lambda^n_i}$ and $\sigma^n_i.$

The following are our main assumptions. 

\begin{assumption}[Conditions on inter-arrival times] \label{arrival}
 
There exist independent sequences of i.i.d. nonnegative random variables $\{\tilde{u}^n_{i,k}: k\in\NN\}, i=\pm 1,$ such that for $k\in\NN$ and $n\in\NN,$
\begin{itemize}
\item[\rm (i)]  $u^n_{i,k} = \frac{\tilde{u}^n_{i,k}}{n};$
\item[\rm (ii)] $\E(\tilde{u}^n_{i,k}) \to \frac{1}{\lambda_i}, \ \var(\tilde{u}^n_{i,k})\to \sigma_i^2$ as $n\to\infty$;
\item[\rm (iii)]  $\{(\tilde{u}^n_{i,1})^2: n\in \NN\}$ is uniformly integrable. 

\end{itemize}
\end{assumption}

\begin{assumption}[Heavy traffic condition]\label{htc}
There exists $c\in \RR$ such that 
\begin{equation}\label{heavytraffic}
\frac{\lambda^n_1-\lambda^n_{-1}}{\sqrt{n}} \to c, \ \mbox{as $n\to\infty.$}
\end{equation}
\end{assumption}

\begin{assumption}[Conditions on patience times]\label{patience}
For $i=\pm 1$, 
$F^n_i(0)=0$, and there exists nonnegative nondecreasing locally Lipschitz continuous functions $H_i$ on $[0,\infty)$ such that for $t\ge 0,$
\bes
\sqrt{n} F^n_i\left(\frac{t}{\sqrt{n}}\right)\to H_i(t), \ \mbox{as $n\to\infty.$}
\ees
In particular, $F^n_i(\frac{t}{\sqrt{n}})\to 0$, as $n\to\infty.$
\end{assumption}

\begin{assumption}[Conditions on hazard rate functions]\label{hazard}
For $i=\pm 1$, $F^n_i$ has density function $f^n_i$ and hazard rate function $h^n_i  = f^n_i/(1-F^n_i)$ on $[0, M^n_i)$, where $M^n_i = \sup\{s\ge 0: F^n(s) < 1\}$. Furthermore, for $t\in [0, M^n_i)$, $\sup_{n\in\NN}\|h^n_i\|_t < \infty$, and $h^n_i(t/\sqrt{n})\to h_i(t)$ as $n\to\infty$, where $h_i$ is a nonnegative measurable function. 
\end{assumption}

\begin{remark}\label{remark}\hfill
\begin{itemize}
\item[\rm (i)] From Assumption \ref{arrival}, we have for $i=\pm 1,$
\begin{align}
\lambda^n_i = \frac{1}{ \E(u^n_{i,1})}, \ \ \mbox{and} \ \ \frac{\lambda^n_i}{n}\to \lambda_i,
\end{align}
and 
\begin{align}
(\sigma^n_i)^2 = \var(u^n_{i,1}), \ \ \mbox{and} \ \ n^2(\sigma^n_i)^2\to \sigma_i^2.
\end{align}
\item[\rm (ii)] The heavy traffic condition in Assumption \ref{htc} implies that $\lambda_1 = \lambda_{-1}$. We are going to let 
\begin{align}\label{lambda}
\lambda = \lambda_1 =\lambda_{-1}.
\end{align}
\item[\rm (iii)] Assumption \ref{patience} is similar to Assumption 3.4 in \cite{Lee11}. We list some examples of the cumulative distribution function satisfying Assumption \ref{patience} (also see Remark 3.4 of \cite{Lee11}).
\begin{itemize}
\item[\rm (a)] Let $F$ be a cumulative distribution function which is right differentiable at $0$, and $F(0)=0$. Then $F^n\equiv F$ satisfies Assumption \ref{patience} with $H(x) = F'(0)x$. Such patience times distribution has been considered in \cite{Man05, Man12, Dai10, DaiHe}.
\item[\rm (b)] Consider $F^n(x) = 1 -\exp\{-\int_0^x h(\sqrt{n} u) du\}$, where $h$ is a nonnegative continuous function. Then $F^n$ satisfies Assumption \ref{patience} with $H(x) = \int_0^x h(u) du$. Here $h^n(x)\equiv h(\sqrt{n}x), x\ge 0$, is the hazard rate function of $F^n$, and such hazard rate scaling of patience times distribution has been studied in \cite{Reed08, ReedTezcan12, kang10}. 
\item[\rm (c)] There are many other distribution functions besides those in (a) and (b) satisfying Assumption \ref{patience}. For example, let 
\[
F^n_i(x) = \begin{cases}
F^n(x), & x\in [0,\delta) \\
1, & x\in [\delta, \infty),
\end{cases}
\]
where $F^n$ is as in (a) or (b), and $\delta>0$. It is easy to see that $F^n_i$ has the same scaling limit function as $F^n$. In fact, Assumption \ref{patience} only requires $F^n_i$ has the scaling limit over a small interval $[0, \mathcal{O}(\frac{1}{\sqrt{n}})].$
\end{itemize}
\item[\rm (iv)] Under Assumption \ref{hazard}, we have $F^n_i(t) = 1 - \exp\{-\int_0^x h^n_i(u) du\}$, and $F^n_i$ satisfies Assumption \ref{patience} with $H_i(x) = \int_0^x h_i(u) du$. Thus the patience times distribution discussed in (iii)(b) is a special case of Assumption \ref{hazard}. This assumption is the same as the main assumption in \cite{kang10}.
\end{itemize}
\end{remark}

\section{Main results}\label{sec:results}
Define the fluid and diffusion-scaled processes as follows. For $i=\pm 1$ and $t\ge 0,$ \\
{\em Fluid scaled processes:}
\bes
 \barq^n(t) = \frac{Q^n(t)}{n}, \; \barn^n_i(t) = \frac{N_i^n(t)}{n}, \ \barg^n_i(t) = \frac{G_i^n(t)}{n}, \ \barr^n_i(t) = \frac{R_i^n(t)}{n},
\ees
{\em Diffusion scaled processes:}
\begin{align*}
&\hatq^n(t) = \frac{Q^n(t)}{\sqrt{n}}, \; \hatn^n_i(t) = \frac{N^n_i(t)-\lambda_i^nt}{\sqrt{n}}, \;  \hatg^n_{i}(t) = \frac{G^n_i(t)}{\sqrt{n}}, \ \hatr^n_{i}(t) = \frac{R^n_{i}(t)}{\sqrt{n}},\\
&  \hatw^n_i(t) = \sqrt{n}W^n_i(t).
\end{align*}
Our main results are presented as follows. Theorem \ref{relation-thm} establishes an asymptotic linear relationship between the diffusion-scaled offered waiting time processes and the diffusion scaled queue length process for both classes. The scaling limit theorem for the queue length process is provided in Theorem \ref{diffusion-thm}, and we derive the stationary distribution of the diffusion limit in Theorem \ref{stability}. All these results are established under Assumptions \ref{arrival}, \ref{htc} and \ref{patience}, and a further condition that $\hatq^n$ is stochastically bounded. At last, we show that $\hatq^n$ is stochastically bounded under Assumptions \ref{arrival}, \ref{htc}, and \ref{hazard}, and consequently, all results in Theorem \ref{relation-thm} -- \ref{stability} hold (see Corollary \ref{cor} for details).

Recall the parameters $\lambda, c,$ and $\sigma_1, \sigma_{-1}$ in \eqref{lambda}, \eqref{heavytraffic}, and Assumption \ref{arrival}, respectively, and the functions $H_1, H_{-1}$ in Assumption \ref{patience}.
\begin{theorem}\label{relation-thm}
Assume that $\hatq^n$ is stochastically bounded and $\hatq^n(0)$ converges weakly to some random variable $q$ as $n\to\infty$. Then under Assumptions \ref{arrival}, \ref{htc} and \ref{patience}, we have 
\bes
\left(\hatw^n_1 - \frac{\hatq^{n,+}}{\lambda}, \ \hatw^n_{-1} - \frac{\hatq^{n,-}}{\lambda}\right) \Go 0, \ \ \mbox{as $n\to\infty.$}
\ees
\end{theorem}

\begin{theorem}\label{diffusion-thm}
Under the same conditions as in Theorem \ref{relation-thm}, we have $\hatq^n \Go Q$ as $n\to\infty$, where for $t\ge 0,$
\begin{equation}\label{sie}
Q(t) = q + \sqrt{\lambda^3(\sigma_1^2+\sigma_{-1}^2)}W(t) + ct - \lambda \int_0^t H_1\left(\frac{Q^+(s)}{\lambda}\right) \D s + \lambda \int_0^t H_{-1}\left(\frac{Q^-(s)}{\lambda}\right) \D s, 
\end{equation}
with $W$ a standard Brownian motion independent of $q$.
\end{theorem}

\begin{theorem}\label{stability} Let $Q$ be as in Theorem \ref{diffusion-thm}. Assume that  
\begin{align*}
\lim_{x\to\infty} H_1(x) > \frac{c}{\lambda}, \ \ \mbox{if $c\ge 0$,}
\end{align*}
and 
\begin{align*}
 \lim_{x\to\infty} H_{-1}(x) > -\frac{c}{\lambda}, \ \ \mbox{if $c\le 0$.}
\end{align*}
Then $Q$ has a unique stationary distribution with density function given by 
\begin{align*}
\pi(\D x) = \begin{cases}
C_0 \exp\left\{- \frac{2}{\lambda^3(\sigma^2_1+\sigma^2_{-1})} \left(-cx + \lambda^2 \int_0^{\frac{x}{\lambda}} H_1(u) \D u  \right) \right\} \D x, & \mbox{if $x \ge 0$}, \\
C_0 \exp\left\{- \frac{2}{\lambda^3(\sigma^2_1+\sigma^2_{-1})} \left(-cx + \lambda^2 \int_0^{-\frac{x}{\lambda}} H_{-1}(u) \D u \right) \right\} \D x, & \mbox{if $x < 0$},
\end{cases}
\end{align*}
where $C_0$ is a positive constant such that $\int_{\mathbb{R}}\pi(\D x) =1.$

\end{theorem}

\begin{corollary}\label{cor}
Under Assumptions \ref{arrival}, \ref{htc} and \ref{hazard},  $\hatq^n$ is stochastically bounded, and if $\hatq^n(0)$ converges weakly to some random variable $q$ as $n\to\infty$, then $\hatq^n \Go Q$ as $n\to\infty$, where 
\bes
Q(t) = q + \sqrt{\lambda^3(\sigma_1^2+\sigma_{-1}^2)}W(t) + ct - \lambda \int_0^t \int_0^{\frac{Q^+(s)}{\lambda}} h_1(u) \D u \D s + \lambda \int_0^t \int_0^{\frac{Q^-(s)}{\lambda}} h_{-1}(u)\D u \D s, \ t\ge 0,
\ees
with $W$ a standard Brownian motion. Furthermore, assume that  
\begin{align*}
\int_0^\infty h_1(u) \D u >  \frac{c}{\lambda}, \ \ \mbox{if $c\ge 0$,}
\end{align*}
and 
\begin{align*}
\int_0^\infty h_{-1}(u) \D u > - \frac{c}{\lambda}, \ \ \mbox{if $c\le 0$.}
\end{align*}
Then $Q$ has a unique stationary distribution with density function given by 
\begin{align*}
\pi(\D x) = \begin{cases}
C_0 \exp\left\{- \frac{2}{\lambda^3(\sigma^2_1+\sigma^2_{-1})} \left(-cx + \lambda \int_0^x \int_0^{\frac{s}{\lambda}} h_1(u) \D u \D s \right) \right\} \D x, & \mbox{if $x \ge 0$}, \\
C_0 \exp\left\{- \frac{2}{\lambda^3(\sigma^2_1+\sigma^2_{-1})} \left(-cx + \lambda \int_0^{-x} \int_0^{\frac{s}{\lambda}} h_{-1}(u) \D u \D s \right) \right\} \D x, & \mbox{if $x < 0$},
\end{cases}
\end{align*}
where $C_0$ is a positive constant such that $\int_{\mathbb{R}}\pi(\D x) =1.$

\end{corollary}

\section{Proofs }\label{proofs}
We first re-characterize the offered waiting times in the following lemma, which will be used in other proofs. Recall from \eqref{owt-1}, \eqref{owt-2} and \eqref{vwt} that for $i=\pm1$ and $k\in\NN,$
\[
w_{i,k}^n = \left( t^n_{-i, k+Q^n_i(0)-R^n_i(t^n_{i,k}-)-Q^n_{-i}(0)+R^n_{-i}(t^n_{i,k}-)} - t^n_{i,k}\right)1_{\{Q^n_{-i}(t^n_{i,k}-)=0\}}.
\]
Recall from Remark \ref{remark_1} that $k-R^n_{-1}(t^n_{-1,k}-)-Q^n_1(0)+R^n_{1}(t^n_{-1,k}-)$ could be negative or zero. 
\begin{lemma}\label{reformulate}
For $k\in\NN,$ assume $k-R^n_{-1}(t^n_{-1,k}-)-Q^n_1(0)+R^n_{1}(t^n_{-1,k}-)>0$. Then for $i=\pm1$, when $Q^n_{-i}(t^n_{i,k}-)=0$,
 \begin{equation}\label{eq:2}
 t_{-i, k+Q^n_i(0)-R^n_i(t^n_{i,k}-)-Q^n_{-i}(0)+R^n_{-i}(t^n_{i,k}-)} \ge t^n_{i,k},
 \end{equation}
and when $Q^n_{-i}(t^n_{i,k}-)>0$,
 \begin{equation}\label{eq:3}
 t^n_{-i, k+Q^n_i(0)-R^n_i(t^n_{i,k}-)-Q^n_{-i}(0)+R^n_{-i}(t^n_{i,k}-)} \le t^n_{i,k}.
 \end{equation}
 If $k-R^n_{-1}(t^n_{-1,k}-)-Q^n_1(0)+R^n_{1}(t^n_{-1,k}-) \le 0$, then  $Q^n_{1}(t^n_{-1,k}-)>0$, and $w^n_{-1,k}=0$.
Consequently, we have for $i=\pm 1$, and $k\in\NN$,
\begin{equation}\label{eq:4}
w^n_{i,k} = \left[t^n_{-i, k+Q^n_i(0)-R^n_i(t^n_{i,k}-)-Q^n_{-i}(0)+R^n_{-i}(t^n_{i,k}-)} - t^n_{i,k}\right]^+.
\end{equation}

\end{lemma}
\begin{proof} Let $i=\pm 1$. For $t\ge 0$, when $Q^n_i(t) =0$, we observe that $R^n_i(t) = G^n_i(t)$, and $Q^n_i(0) + N^n_i(t) - R^n_i(t)$ is the number of the matched pairs that leave the system by time $t$. On the other hand, if $Q^n_i(t)>0$, then $Q^n_i(0) + N^n_i(t) - R^n_i(t)$ will be greater than or equal to the number of the matched pairs that leave the system by time $t$. So when $Q^n_i(t) =0$, we have $Q^n_{-i}(t)\ge 0$, and 
\begin{equation}
Q^n_i(0) + N^n_i(t) - R^n_i(t) \le Q^n_{-i}(0) + N^n_{-i}(t) - R^n_{-i}(t).
\end{equation}
Thus for $k\in\NN$, if  $Q^n_{-i}(t^n_{i,k}-)=0$, we have 
\bes
N^n_{-i}(t^n_{i,k}-)+Q^n_{-i}(0)-R^n_{-i}(t^n_{i,k}-) \le N^n_{i}(t^n_{i,k}-)+Q^n_{i}(0) - R^n_{i}(t^n_{i,k}-) = k-1 +Q^n_{i}(0) - R^n_{i}(t^n_{i,k}-).
 \ees
Thus 
 \begin{align}\label{imp-eqn}
 k +Q^n_{i}(0) - R^n_{i}(t^n_{i,k}-) - Q^n_{-i}(0)+R^n_{-i}(t^n_{i,k}-) \ge N^n_{-i}(t^n_{i,k}-)+1.
 \end{align}
 From \eqref{imp-eqn}, if $k-R^n_{-1}(t^n_{-1,k}-)-Q^n_1(0)+R^n_{1}(t^n_{-1,k}-) \le 0$, then  $Q^n_{1}(t^n_{-1,k}-)$ cannot be zero, and must be positive. Now assume $k-R^n_{-1}(t^n_{-1,k}-)-Q^n_1(0)+R^n_{1}(t^n_{-1,k}-) > 0.$ Again from \eqref{imp-eqn}, we have
  \begin{align*}
 t^n_{-i,k +Q^n_{i}(0) - R^n_{i}(t^n_{i,k}-) - Q^n_{-i}(0)+R^n_{-i}(t^n_{i,k}-)} \ge t^n_{-i,N^n_{-i}(t^n_{i,k}-)+1}\ge t^n_{i,k}.
 \end{align*}
 This shows \eqref{eq:2}. When $Q^n_{-i}(t^n_{i,k}-)>0$, we have $Q^n_{i}(t^n_{i,k}-)=0$, which implies $Q^n_{i}(t^n_{i,k})=0$ and $G^n_i(t^n_{i,k})=G^n_i(t^n_{i,k}-)=R^n_i(t^n_{i,k}-).$ Similar to the above analysis, $N^n_i(t^n_{i,k}) + Q^n_i(0)-G^n_i(t^n_{i,k}) = k+Q^n_i(0)-R^n_i(t^n_{i,k}-)$ is the number of matched pairs that leave the system by time $t^n_{i,k}$. Then we have 
 \begin{align*}
  k+Q^n_i(0)-R^n_i(t^n_{i,k}-) \le N^n_{-i}(t^n_{i,k}) + Q^n_{-i}(0)- R^n_{-i}(t^n_{i,k}) \le N^n_{-i}(t^n_{i,k}) + Q^n_{-i}(0)- R^n_{-i}(t^n_{i,k}-).
 \end{align*}
 Thus
  \begin{align*}
  k+Q^n_i(0)-R^n_i(t^n_{i,k}-) -Q^n_{-i}(0) +R^n_{-i}(t^n_{i,k}-) \le N^n_{-i}(t^n_{i,k}),
 \end{align*}
 and
   \begin{align*}
 t^n_{-i,k+Q^n_i(0)-R^n_i(t^n_{i,k}-) -Q^n_{-i}(0) +R^n_{-i}(t^n_{i,k}-)} \le t^n_{-i,N^n_{-i}(t^n_{i,k})} \le t^n_{i,k}.
 \end{align*}
This shows \eqref{eq:3}. Finally, \eqref{eq:4} follows immediately from \eqref{eq:2} and \eqref{eq:3}. \end{proof}

Let $i=\pm1$. We next introduce the following filtrations. First recall that we assume $Q_1(0)\ge 0$ and $Q_{-1}(0)=0$. We will let $w^n_{1,k} = 0$ when $k < -Q^n_1(0)+1$, and $w^n_{-1,k} = 0$ when $k\le 0$. 
Now define for $k=0, -1, -2, \ldots$, 
\begin{equation}\label{flitration-neg}
\begin{aligned}
\clf^n_{1,k} & = \sigma\left\{Q^n(0); (d^n_{1,l}, w^n_{1,l}), l = k+1, \ldots, 0; R^n_{1,k}(0); \right.\\ 
& \quad \quad \quad \left. t^n_{-1,j}, j =1, \ldots, -k+1-R^n_{1,k}(0)\right\}, \\
\clf^n_{-1,k} & = \sigma\left\{Q^n(0)\right\},
\end{aligned}
\end{equation}
and for $k\in\NN$,
\begin{align}
\clf^n_{i,k} & = \sigma\left\{Q^n(0); t^n_{i,l}, l = 1,2,\ldots, k; N^n_{-i}(s), 0\le s \le t_{i,k}^n; \right. \nonumber\\
\begin{split}
& \quad\quad\quad  (w^n_{i,l}, d^n_{i,l}), l =-Q^n_i(0)+1, \ldots, 0, 1,\ldots, k-1; \\
& \quad\quad\quad  (w^n_{-i,l}, d^n_{-i,l}), l =-Q^n_{-i}(0)+1, \ldots, 0,  1, \ldots, N^n_{-i}(t_{i,k}^n-);
\end{split}\label{filtration}\\
& \quad\quad\quad \left. t^n_{-i,l}, l = 1, 2, \ldots, k+Q_i^n(0) +R_i^n(t_{i,k}^n-) - Q^n_{-i}(0) + R^n_{-i}(t^n_{i,k}-) \right\}. \nonumber
\end{align}
\begin{lemma}\label{filt}
For $k\in\ZZ$, $w^n_{i,k}\in \clf^n_{i,k}$, and $d^n_{i,k}$ is independent of $\clf^n_{i,k}$.
\end{lemma}
\begin{proof}
From \eqref{owt-1} and \eqref{owt-2}, it is clear that $w^n_{i,k}\in \clf^n_{i,k}, k\in\ZZ.$ To see $d^n_{i,k}$ is independent of $\clf^n_{i,k}$, it suffices to show that $d^n_{i,k}$ is independent of the offered waiting times of all customers from both classes who arrive before $t_{i,k}^n$, 
because we assume $Q^n(0)$, $\{u^n_{1,k}: k\in \mathbb{N}\}, \{u^n_{-1,k}: k\in\mathbb{N}\},$ $\{d^n_{1,k}: k\in\mathbb{Z}\},$ and $\{d^n_{-1,k}: k\in\mathbb{Z}\}$ are independent. Fix $k\in\ZZ$. We will use mathematical induction in the following. For the first customer who arrives at the system (prior to time $0$ or after time $0$), assuming he/she is from Class $i$, the offered waiting time is $t_{-i,1}$, which is independent of $d_{i,k}$. Assume that $d_{i,k}$ is independent of the offered waiting times of the first $l$ customers (among all customers of both classes) who arrive at the system before $t_{i,k}^n$.
We next note that for any $l\in\ZZ$ and $i=\pm 1$, the offered waiting time $w_{i,l}$ of the $\left(l1_{\{l>0\}}+(-l+1)1_{\{l\le 0\}}\right)$-th customer of Class $i$ is determined by the offered waiting times of all customers from both classes who arrive before him/her, and other random variables that are independent of $d_{i,k}$. Thus $d_{i,k}$ is also independent of the offered waiting time of the $\left((l+1)1_{\{l>0\}}+(-l+2)1_{\{l\le 0\}}\right)$-th customer (among all customers of both classes) who arrive at the system before $t_{i,k}^n$. The lemma now follows. 
\end{proof}

We divide the rest of the section into four subsections. In Section \ref{sec:decomp}, we decompose $\hatw^n_i$ into several processes, and study the asymptotic behaviors of these processes (see Lemmas \ref{martingale-conv-1}, \ref{martingale-conv-2}, \ref{martingale-conv-3}). In Section \ref{sec:waiting}, using the results from the previous subsection, we prove the $C$-tightness of $(\hatw_1^n, \hatw^n_{-1})$ in Lemma \ref{waiting-tight}, and establish the weak convergence of $(\hatw_1^n, \hatw^n_{-1})$ in Theorem \ref{waiting-convergence}. The weak limit is uniquely characterized by a continuous functional defined in Proposition \ref{unique-map}. Section \ref{sec:reneging} is then devoted to derive the weak convergence of $\hatr^n_i$ and $\hatg^n_i$, and as shown in Lemmas \ref{abandon-5} and \ref{abandon-equal}, they have the same weak limit that can be characterized in terms of the offered waiting time processes. We also provide all the proofs of Theorems \ref{relation-thm}, \ref{diffusion-thm}, and \ref{stability} in this section. Finally, the proof of Corollay \ref{cor} can be found in Section \ref{coro:proof}.

\subsection{Decomposition of $\hatw^n_i$}\label{sec:decomp}
We decompose $\hatw^n_i$ into several stochastic processes, and these processes will be analyzed separately in Lemmas \ref{martingale-conv-1}, \ref{martingale-conv-2}, \ref{martingale-conv-3}. 
\begin{lemma}\label{decomp}
For $i=\pm 1$ and $t\ge 0$, 
\be\label{vw-fluid}\ba
\hatw^n_i(t) & = \left[\hat M^n_{i,1}(t) + \hat M^n_{i,2}(t) - \frac{n}{\lambda^n_{-i}}\hat M^n_{i,3}(\bar N^n_i(t)) +  \frac{n}{\lambda^n_{-i}}\hat M^n_{-i,3}(\bar N^n_{-i}(t)) \right.\\
& \quad\quad   - \frac{\sqrt{n}}{\lambda^n_{-i}} \int_0^t F^n_i(W^n_i(u-)) \D N^n_i(u) + \frac{\sqrt{n}}{\lambda^n_{-i}}   \int_0^t F^n_{-i}(W^n_{-i}(u-)) \D N^n_{-i}(u)  \\
& \quad\quad \left.  + \frac{n}{\lambda^n_{-i}}\left(\frac{1}{\sqrt{n}}+\hatq^n_i(0) -\hatq^n_{-i}(0) -\hat\xi^n_i + \hat\xi^n_{-i}\right) \right]^+,
\ea\ee
where
\begin{align}
\begin{split}
\hat M^n_{i,1}(t) &= \sqrt{n}\biggm[ t^n_{-i, N^n_i(t)+1+Q^n_i(0)-R^n_i(t)-Q^n_{-i}(0)+R^n_{-i}(t)} \\
& \quad - \frac{1}{\lambda_{-i}^n} (N^n_i(t)+1+Q^n_i(0)-R^n_i(t)-Q^n_{-i}(0)+R^n_{-i}(t))\biggm], 
\end{split} \label{mart-1}\\
\hat M^n_{i,2}(t) &= \frac{\sqrt{n}}{\lambda_{-i}^n} (N^n_i(t) - \lambda_{-i}^nt), \label{mart-2}\\
\hat M^n_{i,3}(t) &= \frac{1}{\sqrt{n}} \sum_{k=1}^{[nt]} \left(1_{\{d^n_{i,k} < w^n_{i,k}\}} - F^n_i\left({w_{i,k}^n}\right)\right), \label{mart-3}\\
\hat\xi^n_i & =   \frac{1}{\sqrt{n}} \sum_{k=-Q^n_i(0)+1}^{0} 1_{\{d^n_{i,k} < w^n_{i,k}\}}. \label{mart-4}
\end{align}
\end{lemma}
\begin{proof}
It suffices to show that 
\be\label{total-abandon}\ba
\hat{R}^n_i(t)  & = \hat\xi^n_i + \hat M^n_{i,3}(\bar N^n_i(t))  + \frac{1}{\sqrt{n}} \int_0^t F^n_i(W^n_i(u-)) \D N^n_i(u), \ t\ge 0.
\ea\ee
From Lemma \ref{filt}, for $k\in\NN$, $w^n_{i,k}\in \clf^n_{i,k}$, and $d^n_{i,k}$ is independent of $\clf^n_{i,k}$. Thus for $t\ge 0$, 
\begin{align*}
\hat{R}^n_i(t) & = \frac{1}{\sqrt{n}} \sum_{k=-Q^n_i(0)+1}^{N_i^n(t)} 1_{\{d^n_{i,k}  < w^n_{i,k}\}} \\
&  = \frac{1}{\sqrt{n}} \sum_{k=-Q^n_i(0)+1}^{0} 1_{\{d^n_{i,k}  < w^n_{i,k}\}}  + \frac{1}{\sqrt{n}} \sum_{k=1}^{N_i^n(t)} 1_{\{d^n_{i,k}  < w^n_{i,k}\}} \\
& = \hat\xi^n_i  +  \frac{1}{\sqrt{n}} \sum_{k=1}^{N_i^n(t)} \left(1_{\{d^n_{i,k}  < w^n_{i,k}\}} - \E\left(1_{\{d^n_{i,k}  < w^n_{i,k}\}}|\clf^n_{i,k}\right)\right)  + \frac{1}{\sqrt{n}} \sum_{k=1}^{N^n_i(t)} \E\left(1_{\{d^n_{i,k}  < w^n_{i,k}\}}|\clf^n_{i,k}\right)\\
& = \hat\xi^n_i + \frac{1}{\sqrt{n}} \sum_{k=1}^{N_i^n(t)} \left(1_{\{d^n_{i,k}  < w^n_{i,k}\}} - F^n_i \left({w_{i,k}^n}\right)\right) + \frac{1}{\sqrt{n}} \sum_{k=1}^{N^n_i(t)} F^n_i \left({w_{i,k}^n}\right) \\
& = \hat\xi^n_i +\frac{1}{\sqrt{n}} \sum_{k=1}^{N_i^n(t)} \left(1_{\{d^n_{i,k}  < w^n_{i,k}\}} - F^n_i \left({w_{i,k}^n}\right)\right)  + \frac{1}{\sqrt{n}} \int_0^t F^n_i(W^n_i(u-)) \D N^n_i(u).
\end{align*}
The lemma follows. 
\end{proof}

We first establish the fluid limit of the state process in Lemma \ref{fluid}, and then study the processes in \eqref{mart-1} -- \eqref{mart-4} as $n\to\infty$ in Lemmas \ref{martingale-conv-1}, \ref{martingale-conv-2}, \ref{martingale-conv-3}.

\begin{lemma}\label{fluid} Assume that $\hatq^n$ is stochastically bounded. Then under Assumptions \ref{arrival} and \ref{patience}, we have 
\begin{equation}
(\barn^n_1, \barn^n_{-1}, \barr^n_1, \barr^n_{-1}, W^n_1, W^n_{-1}, \barq^n) \Go (\lambda_1, \lambda_{-1}, 0, 0, 0, 0, 0)\iota,
\end{equation}
where $\iota:[0,\infty)\to [0, \infty)$ is the identity map. 
\end{lemma}
\begin{proof}
Using functional law of large numbers for renewal processes (see Theorem 1 in \cite{iw71}), we have 
\begin{equation}\label{fluid-conv-1}
(\barn^n_1, \barn_{-1}^n) \Go (\lambda_1, \lambda_{-1})\iota.
\end{equation}
Recall that for $t\ge 0$,
\bes
W^n_i(t) = \left(t^n_{-i, N^n_i(t)+1+Q^n_i(0)-R^n_i(t)-Q^n_{-i}(0)+R^n_{-i}(t)} - t\right)1_{\{Q_{-i}^n(t)=0\}}.
\ees
When $Q^n_{-i}(t)=0$, we have $R^n_{-i}(t) =G^n_{-i}(t).$ Thus
\begin{align*}
W^n_i(t) & \le |t^n_{-i, N^n_i(t)+1+Q^n_i(0)-G^n_i(t)-Q^n_{-i}(0)+G^n_{-i}(t)} - t|.
\end{align*}
We first consider $W^n_1$, and note that from \eqref{queue-length}, for $t\ge 0$,
\[
|N^n_1(t)+1+Q^n_1(0)-G^n_1(t)-Q^n_{-1}(0)+G^n_{-1}(t)| = |Q^n(t) +1+ N^n_{-1}(t)| \le |Q^n(t)| + 1 + N^n_{-1}(t).
\] 
Consequently, we have for $T\ge 0,$
\begin{align}
& \sup_{0\le t \le T} W^n_1(t)  \le \sup_{0\le t \le T} \left|\sum_{k=1}^{|Q^n(t)| + 1 + N^n_{-1}(t)} u^n_{-1, k} - t \right| \nonumber\\
& =  \sup_{0\le t \le T} \left|\sum_{k=1}^{|Q^n(t)| + 1 + N^n_{-1}(t)} \left({u}^n_{-1, k} - \frac{1}{\lambda^n_{-1}} \right) + \frac{1}{\lambda^n_{-1}}(|Q^n(t)| + 1) + \frac{1}{\lambda^n_{-1}} (N^n_{-1}(t) - \lambda^n_{-1}t)  \right| \nonumber\\
& \le \sup_{0\le t \le T}  \left|\sum_{k=1}^{|Q^n(t)| + 1 + N^n_{-1}(t)} \left({u}^n_{-1, k} - \frac{1}{\lambda^n_{-1}} \right) \right| + \frac{n}{\lambda^n_{-1}}\|\barq^n\|_t + \frac{1}{\lambda^n_{-1}} + \frac{n}{\lambda^n_{-1}} \frac{\|\hatn^n_{-1}\|_t}{\sqrt{n}}.\label{fluid-waiting}
\end{align}
Using functional central limit theorems for triangular arrays and renewal processes (see again Theorem 1 in \cite{iw71}), under Assumption \ref{arrival}, 
\begin{align}\label{iw-thm}
\left({\sqrt{n}}\sum_{k=1}^{\lfloor n\cdot \rfloor} \left({u}^n_{-1, k} - \frac{1}{\lambda^n_{-1}} \right), \ \ \hatn^n_{-1}\right)\Go (\sigma_{-1}B, \ \ -\sigma_{-1}\lambda_{-1}^{3/2}B),
\end{align}
where $B$ is a standard Brownian motion. Noting that $\barq^n \Go 0$, and applying the random time change theorem to the first term in \eqref{fluid-waiting}, we have $W_1\Go 0.$ Similarly, we can show the same convergence result for $W^n_{-1}$. Thus 
\begin{equation}\label{fluid-conv-3}
W^n_i \Go 0.
\end{equation}
Next from \eqref{total-abandon}, for $t\ge 0,$
\begin{align*}
\bar{R}^n_i(t)  & = \frac{1}{\sqrt{n}}\hat\xi^n_i + \frac{1}{\sqrt{n}}\hat M^n_{i,3}(\bar N^n_i(t))  + \frac{1}{n} \int_0^t F^n_i(W^n_i(u-)) \D N^n_i(u), \ t\ge 0.
\end{align*}
It is clear that 
\begin{equation}\label{fluid-reneging-1}
\frac{1}{\sqrt{n}}\hat\xi^n_i  \le |\barq^n(0)|\Go 0.
\end{equation} 
We next consider $\frac{1}{\sqrt{n}}\hatm^n_{i,3}$ and note that $\{\frac{1}{\sqrt{n}}\hatm^n_{i,3}(t)\}_{t\ge 0}$ is a $\{\clf^n_{i,t}\vee  \sigma(d^n_{i,[nt]})\}_{t\ge 0}$ martingale, where  $\{\clf_{i,k}^n\}_{k\in\NN}$ is defined in \eqref{filtration}. From Doob's inequality for martingales, we have for $T\ge 0,$ 
\bes\label{doob}
\E\left(\sup_{0\le t\le T}\left(\frac{1}{\sqrt{n}}\hatm^n_{i,3}(t)\right)^2\right) \le 4 \E\left(\left(\frac{1}{\sqrt{n}}\hatm^n_{i,3}(T)\right)^2\right) = 4 \E\left([\frac{1}{\sqrt{n}}\hatm^n_{i,3}](T)\right) \le \frac{8nT}{n^2} \to 0.
\ees
From the random time change theorem, we have 
\begin{equation}\label{fluid-reneging-2}
\frac{1}{\sqrt{n}}\hatm^n_{i,3}(\barn^n_i(\cdot)) \Go 0.
\end{equation}
Next using Assumption \ref{patience}, we have that for $\delta>0,$
\begin{align}
&\limsup_{n\to\infty} \PP\left(\frac{1}{n} \int_0^t F^n_i(W^n_i(u-)) \D N^n_i(u) > \delta \right)\label{fluid-reneging-4}\\
& =\lim_{K\to\infty}\limsup_{n\to\infty} \PP\left(\frac{1}{n} \int_0^t F^n_i(W^n_i(u-)) \D N^n_i(u) > \delta, \|\hatw^n_i\|_{t} \le K \right) + \lim_{K\to\infty}\limsup_{n\to\infty}\PP( \|\hatw^n_i\|_{t} > K). \nonumber
\end{align}
From \eqref{fluid-waiting}, \eqref{iw-thm}, and the assumption that $\hatq^n$ is stochastically bounded, we see that $\hatw^n_i$ is also stochastically bounded, and so 
\begin{align*}
\lim_{K\to\infty}\limsup_{n\to\infty}\PP( \|\hatw^n_i\|_{t} > K) =0. 
\end{align*}
Therefore, \eqref{fluid-reneging-4} can be bounded by 
\[
\lim_{K\to\infty}\limsup_{n\to\infty}  \PP\left( \barn^n_i(t) > \frac{\delta}{F^n_i(K/\sqrt{n})} \right)   =0.
\]
This shows that 
\begin{equation}\label{fluid-reneging-3}
\frac{1}{n} \int_0^\cdot F^n_i(W^n_i(u-)) \D N^n_i(u) \Go 0.
\end{equation}
Combining \eqref{fluid-reneging-1}, \eqref{fluid-reneging-2}, and \eqref{fluid-reneging-3}, we have 
\begin{equation}\label{fluid-conv-4}
\barr^n_i \Go 0.
 \end{equation}
 The lemma follows from \eqref{fluid-conv-1}, \eqref{fluid-conv-3}, and \eqref{fluid-conv-4}.
\end{proof}

\begin{lemma}\label{martingale-conv-1} Assume that $\hatq^n$ is stochastically bounded. Let $\hatm^n_{i,1}$ and $\hatm^n_{i,2}$ be as in \eqref{mart-1} and \eqref{mart-2}. Then under Assumptions \ref{arrival}, \ref{htc} and \ref{patience},
\[
(\hatm^n_{1,1}, \hatm^n_{1,2}, \hatm^n_{-1,1}, \hatm^n_{-1,2}) \Go (0, \frac{c}{\lambda}, 0, -\frac{c}{\lambda})\iota + (Y_{-1}, - Y_1, Y_1, - Y_{-1}), 
\]
where $\iota: [0,\infty)\to [0,\infty)$ is the identity map, and $Y_1(t) = \sigma_1 B_1(\lambda t)$ and $Y_{-1}(t) = \sigma_{-1} B_{-1}(\lambda t)$ with $B_1$ and $B_{-1}$ being two independent standard Brownian motions.
\end{lemma}
\begin{proof}
Define 
\be\label{iid-sum}
\tilde M^n_{i}(t) = \sqrt{n} \sum_{k=1}^{\lfloor \lambda^n_it \rfloor} \left(u^n_{i,k} - \frac{1}{\lambda^n_{i}}\right), \ t\ge 0.
\ee
Using Theorem 1 in \cite{iw71}, under Assumptions \ref{arrival} and \ref{htc}, 
\be\label{fclt}
(\tilde M^n_1, \tilde M^n_{-1}, \frac{1}{\lambda} \hatn^n_1, \frac{1}{\lambda} \hatn^n_{-1})\Go (Y_1, Y_{-1}, -Y_1, -Y_{-1}),
\ee 
where  $Y_1(t) = \sigma_1 B_1(\lambda t)$ and $Y_{-1}(t) = \sigma_{-1} B_{-1}(\lambda t)$ with $B_1$ and $B_{-1}$ being two independent standard Brownian motions. We next note that for $i=\pm 1$ and $t\ge 0,$
\begin{align}\label{fclt-au-1}
\hatm^n_{i,1}(t)  = \tilde M^n_{-i}\left(\frac{n}{\lambda_{-i}^n}(\barn^n_i(t)+\frac{1}{n}+\barq^n_i(0)-\barr^n_i(t)-\barq^n_{-i}(0)+\barr^n_{-i}(t))\right),
\end{align}
and 
\be\label{fclt-au-2}
\hatm^n_{i,2}(t) = \frac{n}{\lambda^n_{-i}}\hatn^n_i(t) + \frac{n}{\lambda^n_{-i}}\frac{(\lambda^n_i -\lambda^n_{-i})t}{\sqrt{n}}.
\ee
The lemma follows by applying the random time change theorem, Lemma \ref{fluid}, and Assumptions \ref{arrival} and \ref{htc} to $(\hatm^n_{1,1}, \hatm^n_{1,2}, \hatm^n_{-1,1}, \hatm^n_{-1,2})$. 
\end{proof}

The following lemma on martingale convergence (see  \cite{pang2007} or \cite{daihe2010}) will be used in the proof of Lemma \ref{martingale-conv-2}.

\begin{lemma}\label{martingale-convergence}
Suppose $\{M^n(t); t\ge 0\}$ is a local martingale w.r.t some filtrations, and for $t\ge 0$,
\bes
\E\left( \sup_{0\le s \le t} |M^n(s)-M^n(s-)|\right) \to 0, \ \mbox{and} \ [M^n]_t \Go 0, \ \mbox{as $n\to\infty.$}
\ees
Then $M^n \Go 0$ as $n\to \infty.$
\end{lemma}

\begin{lemma}\label{martingale-conv-2} Assume that $\hatq^n$ is stochastically bounded. Let $\hatm^n_{i,3}$ be as in \eqref{mart-3}. Then under Assumption \ref{arrival}, \ref{htc}, and \ref{patience}, $\hatm^n_{i,3}\Go 0. $
\end{lemma}
\begin{proof}
Recall that $\{\hatm^n_{i,3}(t)\}_{t\ge 0}$ is an $\{\clf^n_{i,t}\vee \sigma(d^n_{i,[nt]})\}_{t\ge 0}$ martingale, where $\clf^n_{i,t}$ is defined in \eqref{filtration}, and its quadratic variation is 
\bes
[\hatm^n_{i,3}]_t = \frac{1}{n} \sum_{k=1}^{[nt]} \left(1_{\{w^n_{i,k}\ge d^n_{i,k}\}} - F^n_i\left({w_{i,k}^n}\right)\right)^2.
\ees
We next observe that for $t\ge 0,$
\begin{align*}
\E([\hatm^n_{i,3}]_t) & = \frac{1}{n} \sum_{k=1}^{[nt]} \E\left(1_{\{w^n_{i,k}\ge d^n_{i,k}\}} - F^n_i\left({w_{i,k}^n}\right)\right)^2 \\
& = \frac{1}{n} \sum_{k=1}^{[nt]} \E\left(1_{\{w^n_{i,k}\ge d^n_{i,k}\}} - \E(1_{\{w^n_{i,k}\ge d^n_{i,k}\}} | \clf^n_{i,k})\right)^2 \\
& = \frac{1}{n} \sum_{k=1}^{[nt]} \left[\E\left(1_{\{w^n_{i,k}\ge d^n_{i,k}\}}\right) -  \E\left(\E^2(1_{\{w^n_{i,k}\ge d^n_{i,k}\}} | \clf^n_{i,k})\right)\right] \\
& \le \frac{1}{n} \sum_{k=1}^{[nt]} \E\left(1_{\{w^n_{i,k}\ge d^n_{i,k}\}} \right) \\
& = \E\left(\frac{1}{n} \sum_{k=1}^{[nt]} 1_{\{w^n_{i,k}\ge d^n_{i,k}\}} 1_{\{N^n_i(2t/\lambda) \ge n t\}} \right) +\E\left( \frac{1}{n} \sum_{k=1}^{[nt]} 1_{\{w^n_{i,k}\ge d^n_{i,k}\}} 1_{\{N^n_i(2t/\lambda) < n t\}}\right) \\
& \le \E\left(\frac{1}{n} \sum_{k=1}^{N^n_i(2t/\lambda)} 1_{\{w^n_{i,k}\ge d^n_{i,k}\}} \right) + t \PP(N^n_i(2t/\lambda) < n t) \\
& = \E(\barr^n_i(2t/\lambda)) + t \PP(\barn^n_i(2t/\lambda) < t).
\end{align*}
Noting that $\barn^n_i(2t/\lambda)\Go 2t$ as $n\to\infty$, thus $\PP(\barn^n_i(2t/\lambda_i) < t)\to 0.$ Next from Lemma \ref{fluid}, we have $\barr^n_i \Go 0$, and we further observe that $\barr^n_i$ is uniformly integrable (which follows from the uniform integrability of $\barn^n_i$). Thus $\E(\barr^n_i(2t/\lambda))\to 0$. 
Finally, the result follows from Lemma \ref{martingale-convergence}.
\end{proof}

\begin{lemma}\label{martingale-conv-3}
Assume $\hatq^n(0)$ is stochastically bounded. Let $\hat\xi_i$ be as in \eqref{mart-4}. Then under Assumptions \ref{arrival}, \ref{patience}, we have $\hat\xi^n_i \Go 0.$
\end{lemma}
\begin{proof} Noting that we assume $Q_1(0)\ge 0$ and $Q_{-1}(0)=0$, so $\hat\xi^n_{-1} \equiv 0$, and it only needs to show $\hat\xi^n_1\Go 0.$ Recall $\{\clf^n_{1,k}: k =0, -1, -2, \ldots\}$ defined in \eqref{flitration-neg}. From Lemma \ref{filt}, $w^n_{1,k}\in \clf^n_{1,k}$ and $d^n_{1,k}$ is independent of $\clf^n_{1,k}, k\in\ZZ$. 
 Define for $t\ge 0$, 
\bes
\tilde M^n(t) = \frac{1}{\sqrt{n}} \sum_{k=-\lfloor nt \rfloor+1}^{0} [1_{\{d^n_{1,k}  \le w^n_{1,k}\}} - F^n_1(w^n_{1,k})]
\ees
We see that $\{\tilde M^n(t)\}_{t\ge 0}$ is a $\{\clf^n_{1,-\lfloor nt \rfloor} \vee \sigma(d^n_{1,-\lfloor nt \rfloor})\}_{t\ge 0}$ martingale, and 
\begin{align*}
[\tilde M^n]_t & = \frac{1}{n} \sum_{k=-\lfloor nt \rfloor+1}^{0} [1_{\{d^n_{1,k}  \le w^n_{1,k}\}} - F^n_1(w^n_{1,k})]^2, \ t\ge 0.
\end{align*}
It is clear that $\{[\tilde M^n]_t\}_{t\ge 0}$ is $C$-tight and from Theorem VI.4.13 of \cite{js03}, we conclude that $\{\tilde M^n(t)\}_{t\ge 0}$ is also $C$-tight. Using the fact that $\barq^n(0)\Go 0$, we obtain that 
\begin{equation}\label{part-1}
\tilde M^n(\barq^n(0)) \Go 0, \ \mbox{as $n\to\infty.$}
\end{equation}
Thus it suffices to show that for $\delta> 0$,  
\begin{align*}
\lim_{n\to\infty}\PP\left(\frac{1}{\sqrt{n}} \sum_{k=-Q^n(0)+1}^{0} F^n_1(w^n_{1,k}) > \delta\right) = 0.
\end{align*}
We observe that for all $k=-Q^n(0)+1, \ldots, -1, 0$, 
\begin{align*}
\sqrt{n} w^n_{1,k} & \le \sqrt{n} t^n_{-1, Q^n(0)}  = \sqrt{n}\sum_{l=1}^{Q^n(0)} u^n_{-1,l} = \tilde M^n_{-1}(Q^n(0)/\lambda^n_{-1}) + \frac{n}{\lambda^n_{-1}}\hatq^n(0), 
\end{align*}
where $\tilde M^n_{-1}$ is defined in \eqref{iid-sum}. Thus for each $k=-Q^n(0)+1, \ldots, -1, 0$, $\sqrt{n}w^n_{1,k}$ is stochastically bounded, and for $\delta > 0,$
\begin{align*}
&\limsup_{n\to\infty} \PP\left(\frac{1}{\sqrt{n}} \sum_{k=-Q^n(0)+1}^{0} F^n_1(w^n_{1,k}) > \delta\right)\\
& \le  \lim_{K\to\infty}\limsup_{n\to\infty}\PP\left(\frac{1}{\sqrt{n}} \sum_{k=-Q^n(0)+1}^{0} F^n_1(w^n_{1,k}) > \delta, \ \sqrt{n}w^n_{1,k}\le K\right) \\
& \quad +\lim_{K\to\infty}\limsup_{n\to\infty} \PP\left(\sqrt{n}w^n_{1,k}>K\right)\\
& \le\lim_{K\to\infty} \limsup_{n\to\infty} \PP\left(\sqrt{n}F^n_1(K/\sqrt{n}) \barq^n(0) > \delta) \right) \\
& \quad+\lim_{K\to\infty} \limsup_{n\to\infty}\PP\left(\sqrt{n}w^n_{1,k}>K\right)\\
& =0. 
\end{align*}
Combining the above convergence and \eqref{part-1}, we have $\hat\xi_1^n \Go 0.$
\end{proof}

\subsection{Weak convergence of $(\hatw^n_1, \hatw^n_{-1})$}\label{sec:waiting}

We prove the $C$-tightness of $(\hatw_1^n, \hatw^n_{-1})$ in Lemma \ref{waiting-tight}, and establish the weak convergence of $(\hatw_1^n, \hatw^n_{-1})$ in Theorem \ref{waiting-convergence}. The weak limit is uniquely characterized by a continuous functional defined in Proposition \ref{unique-map}. 

\begin{lemma}\label{waiting-tight}
Assume that $\hatq^n$ is stochastically bounded. Then under Assumptions \ref{arrival}, \ref{htc} and \ref{patience}, $(\hatw^n_1, \hatw^n_{-1})$ is $C$-tight.
\end{lemma}
\begin{proof}
From \eqref{fluid-waiting} and \eqref{iw-thm} in the proof of Lemma \ref{fluid}, we know that $\hatw^n_i$ is stochastically bounded. Now from \eqref{vw-fluid}, for $0\le s \le t <\infty,$
\begin{align*}
|\hatw^n_i(t) -\hatw^n_i(s)| & \le |\hatm^n_{i,1}(t) -\hatm^n_{i,1}(s)| + |\hatm^n_{i,2}(t) -\hatm^n_{i,2}(s)| + \frac{n}{\lambda^n_{-i}} |\hat M^n_{i,3}(\bar N^n_i(t)) -\hat M^n_{i,3}(\bar N^n_i(s))| \\
& \quad + \frac{n}{\lambda^n_{-i}} |\hatm^n_{-i,3}(\bar N^n_{-i}(t)) -\hatm^n_{-i,3}(\bar N^n_{-i}(s))| \\
& \quad + \frac{\sqrt{n}}{\lambda^n_{-i}} \int_s^t F^n_i(W^n_i(u-)) \D N^n_i(u) + \frac{\sqrt{n}}{\lambda^n_{-i}}   \int_s^t F^n_{-i}(W^n_{-i}(u-)) \D N^n_{-i}(u).
\end{align*}
From Lemmas \ref{fluid}, \ref{martingale-conv-1} and \ref{martingale-conv-2}, $(\hatm^n_{i,1}(\cdot), \hatm^n_{i,2}(\cdot), \hatm^n_{i,3}(\barn^n_{i}(\cdot)), \hatm^n_{-i,3}(\barn^n_{-i}(\cdot)))$ are weakly convergent. We next note that for $\delta>0,$
\begin{align*}
& \PP\left(\frac{\sqrt{n}}{\lambda^n_{-i}} \int_s^t F^n_i(W^n_i(u-)) \D N^n_i(u) > \delta \right)\\
& = \PP\left(\frac{1}{\lambda^n_{-i}} \int_s^t \sqrt{n}F^n_i(W^n_i(u-)) \D N^n_i(u) > \delta, \|\hatw^n_i\|_{s,t} \le K \right) + \PP( \|\hatw^n_i\|_{s,t} > K) \\
& \le \PP\left(\sqrt{n}F^n_i(K/\sqrt{n}) \ \frac{n}{\lambda^n_{-i}} \left(\barn^n_i(t) -\barn^n_i(s)\right)  > \delta\right) + \PP( \|\hatw^n_i\|_{t} > K).
\end{align*}
Using Assumption \ref{patience} and noting that $\barn^n_i$ is $C$-tight and $\hatw^n_i$ is stochastically bounded, we have that for $\delta>0,$
\begin{align*}
& \lim_{(t-s)\downarrow 0}\limsup_{n\to\infty}  \PP\left(\frac{\sqrt{n}}{\lambda^n_{-i}} \int_s^t F^n_i(W^n_i(u-)) \D N^n_i(u) > \delta \right)  \\
& \le \lim_{K \to\infty}\lim_{(t-s)\downarrow 0}\limsup_{n\to\infty} \PP\left(\sqrt{n}F^n_i(K/\sqrt{n}) \ \frac{n}{\lambda^n_{-i}} \left(\barn^n_i(t) -\barn^n_i(s)\right)  > \delta\right) \\
& \quad + \lim_{K \to\infty}\limsup_{n\to\infty} \PP( \|\hatw^n_i\|_{t} > K) \\
& = 0.
\end{align*}
The result follows. 
\end{proof}

\begin{proposition}\label{unique-map}
Let $H_1, H_{-1}: [0,\infty)\to[0,\infty)$ be nonnegative locally Lipschitz continuous functions as in Assumption \ref{patience}. 
\begin{itemize}
\item[\rm (i)] Given $x\in D([0;\infty), \RR)$, there exists a unique pair of $(w_1, w_{-1})$ such that $w_i \in D([0;\infty), \RR_+), i=\pm 1,$ and for $t\ge 0,$
\begin{align}
w_1(t) & = \left[x(t) - \int_0^t H_1(w_1(s))\D s + \int_0^t H_{-1}(w_{-1}(s))\D s\right]^+,\label{eq:51} \\
w_{-1}(t) & = \left[x(t) - \int_0^t H_1(w_1(s))\D s + \int_0^t H_{-1}(w_{-1}(s))\D s\right]^-. \label{eq:52}
\end{align}
\item[\rm (ii)] Define the functionals $\Psi_1, \Psi_{-1}: D([0,\infty); \RR) \to D([0,\infty); \RR_+)$ by 
\begin{equation}\label{waiting-limit}
(\Psi_1, \Psi_{-1})(x) = (w_1, w_{-1}).
\end{equation}
Then $(\Psi_1, \Psi_{-1})$ is continuous on $D([0,\infty); \RR)$ with Skorohod $J_1$-topology.
\end{itemize}
\end{proposition}
\begin{proof}
See Appendix.
\end{proof}

\begin{theorem}\label{waiting-convergence} 
Assume that $\hatq^n(0)$ converges weakly to some random variable $q$, and $\hatq^n$ is stochastically bounded. Then under Assumptions \ref{arrival}, \ref{htc}, and \ref{patience}, 
\bes
(\hatw^n_1, \hatw^n_{-1}) \Go (\Psi_1, \Psi_{-1})(X),
\ees
where $X$ is a Brownian motion with drift $\frac{c}{\lambda}$ and variance $\lambda(\sigma_1^2+\sigma_{-1}^2)$, and initial value $X(0) = \frac{q}{\lambda}.$
\end{theorem}
\begin{proof}
From \eqref{vw-fluid}, for $i=\pm1$ and $t\ge 0,$
\begin{align*}
\hatw^n_i(t) & = \left[\hat M^n_{i,1}(t) + \hat M^n_{i,2}(t) - \frac{n}{\lambda^n_{-i}}\hat M^n_{i,3}(\bar N^n_i(t)) +  \frac{n}{\lambda^n_{-i}}\hat M^n_{-i,3}(\bar N^n_{-i}(t)) \right.\\
& \quad\quad   - \frac{\sqrt{n}}{\lambda^n_{-i}} \int_0^t F^n_i(W^n_i(u-)) \D N^n_i(u) + \frac{\sqrt{n}}{\lambda^n_{-i}}   \int_0^t F^n_{-i}(W^n_{-i}(u-)) \D N^n_{-i}(u)  \\
& \quad\quad \left.  + \frac{n}{\lambda^n_{-i}}\left(\frac{1}{\sqrt{n}}+\hatq^n_i(0) -\hatq^n_{-i}(0) -\hat\xi^n_i + \hat\xi^n_{-i}\right) \right]^+.
\end{align*}
Let 
\begin{align*}
\hatm^n_i(t) & = \hat M^n_{i,1}(t) + \hat M^n_{i,2}(t) - \frac{n}{\lambda^n_{-i}}\hat M^n_{i,3}(\bar N^n_i(t)) +  \frac{n}{\lambda^n_{-i}}\hat M^n_{-i,3}(\bar N^n_{-i}(t))  + \frac{n}{\lambda^n_{-i}}\left(\frac{1}{\sqrt{n}} -\hat\xi^n_i + \hat\xi^n_{-i}\right).
\end{align*}
From Lemmas \ref{martingale-conv-1}, \ref{martingale-conv-2}, and \ref{martingale-conv-3},  we have 
\be\label{eq:w51}
(\hatm^n_1, \hatm^n_{-1}) \Go (B_1, -B_1), 
 \ee 
where $B_1$ is a Brownian motion with drift $\frac{c}{\lambda}$ and variance $\lambda(\sigma^2_1+\sigma^2_{-1}).$
Next Lemma \ref{waiting-tight} establishes the $C$-tightness of $(\hatw^n_1, \hatw^n_{-1})$. Let $(\tilde W_1, \tilde W_{-1})$ be a weak limit of $(\hatw^n_1, \hatw^n_{-1})$ along with a subsequence $\{n_l\}_{l\in\NN}$. Using Skorohod representation theorem, we can assume $(\hatw^{n_l}_1, \hatw^{n_l}_{-1}, \hatm^{n_l}_{1},$ $\hatm^{n_l}_{-1}, \barn^{n_l}_1, \barn^{n_l}_{-1})$ converges to $(\tilde W_1, \tilde W_{-1}, B_1, -B_1, \lambda_1\iota, \lambda_{-1}\iota)$ almost surely and uniformly in compact sets of $[0,\infty).$ Define for $i=\pm1$ and $t\ge 0,$
\begin{align}
\epsilon_i^n(t) = \frac{n}{\lambda^n_{-i}} \int_0^t \sqrt{n}F^n_i(\hatw^n_i(u-)/\sqrt{n}) \D \barn^n_i(u) - \int_0^t H_i(\hatw^n_i(u)) \D u.
\end{align}
From Lemma 2.4 in \cite{daiwill96}, we have 
\begin{align}\label{eq:w52}
\epsilon_i^{n_l} \Go 0, \ \ \mbox{as $l\to\infty.$}
\end{align}
We can now rewrite the waiting time process as follows.
\begin{align*}
\hatw^n_i(t) & = \left[\hat M^n_{i}(t) + \frac{n}{\lambda^n_{-i}}\left(\hatq^n_i(0) -\hatq^n_{-i}(0) \right)  - \epsilon_i^n(t) + \epsilon_{-i}^n(t)   -\int_0^t H_i(\hatw^n_i(u))\D u \right. \\
&\quad \left.+ \int_0^t H_{-i}(\hatw^n_{-i}(u))\D u \right]^+.
\end{align*}
Using Proposition \ref{unique-map} (ii) and \eqref{eq:w51}, \eqref{eq:w52}, we have 
\be
(\tilde W_1, \tilde W_{-1}) \overset{d}{=} (\Psi_1, \Psi_{-1})(X),
\ee
where $X$ is a Brownian motion with drift $\frac{c}{\lambda}$, variance $\lambda(\sigma_1^2+\sigma_{-1}^2)$, and initial value $X(0) = q/\lambda.$ Finally, from the uniqueness of $(\Psi_1, \Psi_{-1})$ in Proposition \ref{unique-map} (i), we have $(\hatw^n_1, \hatw^n_{-1})\Go (\Psi_1, \Psi_{-1})(X).$
\end{proof}

\subsection{Weak convergence of $\hatr^n_i$ and $\hatg^n_i$}\label{sec:reneging}

The following two lemmas show that both $\hatr^n_i$ and $\hatg^n_i$ converge to 
\[
\lambda_i \int_0^\cdot H_i(\hatw_i(s))\D s,
\]
where $\hatw_i = \Psi_i(X)$, with $X$ defined in Theorem \ref{waiting-convergence}.

\begin{lemma}\label{abandon-5}  Assume that $\hatq^n$ is stochastically bounded. Then under Assumptions \ref{arrival}, \ref{htc}, and \ref{patience}, 
\bes
\hatr^n_i(\cdot) - \lambda_i \int_0^\cdot H_i(\hatw^n_i(s))\D s \Go 0.
\ees
\end{lemma}
\begin{proof}
From \eqref{total-abandon}, for $t\ge 0$, 
\begin{align*}
&\left|\hatr^n_i(t) - \lambda_i \int_0^t H_i(\hatw^n_i(s))\D s \right| \\
& \le \hat\xi^n_i +  |\hatm^n_{i,3}(t)| +  \left|\int_0^t \sqrt{n} F^n_i(\hatw^n_i(u-)/\sqrt{n})\D \barn^n_i(u) - \lambda_i \int_0^t H_i(\hatw^n_i(s))\D s \right|.
\end{align*}
The result follows from Lemmas \ref{martingale-conv-1}, \ref{martingale-conv-2}, and \eqref{eq:w52}.
\end{proof}

\begin{lemma} \label{abandon-equal}
Assume that $\hatq^n$ is stochastically bounded. Then under Assumptions \ref{arrival}, \ref{htc}, and \ref{patience}, 
\bes
(R^n_1 - G^n_1, \ R^n_{-1} -G^n_{-1}) \Go 0.
\ees
\end{lemma}
\begin{proof} The proof idea is similar to the proof of Proposition 4.1 in \cite{daihe2010}. We first show that 
\be\label{finish}
w^n_{i,k}+t^n_{i,k}\le w^n_{i,l}+t^n_{i,l},
\ee 
when $1\le k\le l$, or $l\le k \le 0$, or $k = -Q^n(0)+1, l=1$. 
Recall that $t^n_{i,k}=0$ when $k\le 0$. From \eqref{waiting-negative}, it is clear that $w^n_{1,k}\le w^n_{1,l}$ for $l\le k \le 0$, and so \eqref{finish} holds for $l\le k \le 0$. For $k=-Q^n(0)+1$ and $l=1$, we have 
\begin{align*}
w^n_{1,-Q^n(0)+1} & = t^n_{-1, Q^n_1(0)  - R_{1,-Q^n(0)+1}(0)}, \\
w^n_{1,1} & = \left(t^n_{-1, 1+Q^n_1(0)-R^n_1(t^n_{1,1}-)+R^n_{-1}(t^n_{1,1}-)} - t^n_{1,1}\right)^+.
\end{align*}
We note that $R^n_{1,-Q^n(0)+1} +1 \ge R^n_1(t^n_{1,1}-)$, and so 
\bes
t^n_{-1, Q^n_1(0)  - R_{1,-Q^n(0)+1}(0)} \le t^n_{-1, 1+Q^n_1(0)-R^n_1(t^n_{1,1}-)+R^n_{-1}(t^n_{1,1}-)}. 
\ees
Thus it is clear that \eqref{finish} holds for $k=-Q^n(0)+1$ and $l=1$. Consider now $k\in \NN$. If $w^n_{i,k}=0$, then \eqref{finish} holds clearly. Suppose now that $w^n_{i,k}>0.$ Then 
\bes
w^n_{i,k} + t^n_{i,k} = t^n_{-i, k+Q^n_i(0)-R^n_i(t^n_{i,k}-)-Q^n_{-i}(0)+R^n_{-i}(t^n_{i,k}-)}. 
\ees
Noting that $R^n_i(t^n_{i,l}-) - R^n_i(t^n_{i,k}-) \le l-k$, we have 
\bes
 t^n_{-i, k+Q^n_i(0)-R^n_i(t^n_{i,k}-)-Q^n_{-i}(0)+R^n_{-i}(t^n_{i,k}-)} \le t^n_{-i, l + Q^n_i(0)-R^n_i(t^n_{i,l}-)-Q^n_{-i}(0)+R^n_{-i}(t^n_{i,l}-)}.
\ees
If $w^n_{i,l} >0$, then 
\bes
w^n_{i,l} + t^n_{i,l} = t^n_{-i, l + Q^n_i(0)-R^n_i(t^n_{i,l}-)-Q^n_{-i}(0)+R^n_{-i}(t^n_{i,l}-)},
\ees
and if $w^n_{i,l}=0$, then from Lemma \ref{reformulate},
\bes
w^n_{i,l} + t^n_{i,l} =  t^n_{i,l} \ge t^n_{-i, l + Q^n_i(0)-R^n_i(t^n_{i,l}-)-Q^n_{-i}(0)+R^n_{-i}(t^n_{i,l}-)}.
\ees
Thus \eqref{finish} holds for $1\le k\le l$. For $t\ge 0$, define
\bes
\tau^n_i(t) = \inf\{s\ge 0: s + W^n_i(s) > t\}.
\ees
Observing that $w^n_{i,k}+t^n_{i,k} = W^n_i(t^n_{i,k}-)+t^n_{i,k} \le t$ for all $t^n_{i,k}<\tau^n_i(t)$, each customer arriving before time $\tau^n_i(t)$ should have left the system by time $t$. This says $R^n_i((\tau^n_i(t)-1/n)^+) \le G^n(t)$ for $t\ge 0$. So we have 
\bes
R^n_i((\tau^n_i(t)-1/n)^+) \le G^n_i(t) \le R^n_i(t), \ t\ge 0.
\ees
Next for $t\ge 0$, we note that $\tau^n_i(t)\le t$. We then have for $T>0,$
\begin{align}
&\sup_{0\le t\le T} \left| \hatr^n_i(t) -\hatr^n_i((\tau^n_i(t)-1/n)^+)\right| \nonumber \\
& \le \sup_{0\le t\le T} \left|\hatr^n_i(\tau^n_i(t) + W^n_i(\tau^n_i(t))) -\hatr^n_i((\tau^n_i(t)-1/n)^+)\right| \nonumber\\
& \le  \sup_{0\le t\le T} \left|\hatr^n_i(\tau^n_i(t) + W^n_i(\tau^n_i(t)))  - \lambda_i\int_0^{\tau^n_i(t) + W^n_i(\tau^n_i(t))} H_i(\hatw^n_i(s))\D s \right| \nonumber \\
& \quad + \sup_{0\le t\le T} \left|\hatr^n_i((\tau^n_i(t)-1/n)^+)  - \lambda_i\int_0^{(\tau^n_i(t)-1/n)^+} H_i(\hatw^n_i(s))\D s \right|  \\
& \quad + \lambda_i  \sup_{0\le t\le T}\int_{(\tau^n_i(t)-1/n)^+}^{\tau^n_i(t) + W^n_i(\tau^n_i(t))} H_i(\hatw^n_i(s))\D s \nonumber\\
& \le 2 \sup_{0\le s \le T+ \sup_{0\le u \le T} W^n_i(u)} \left|\hatr^n_i(s) - \lambda_i \int_0^s H_i(\hatw^n_i(v))\D v \right| \label{eq:r51} \\
& \quad  + \lambda_i \sup_{0\le s \le T+ \sup_{0\le u \le T} W^n_i(u)} H_i(\hatw^n_i(s)) \cdot \sup_{0\le s \le T} \left[ s+ W^n_i(s) - (s-1/n)^+\right]. \label{eq:r52}
\end{align}
From Lemmas \ref{abandon-5} and \ref{fluid}, the term in \eqref{eq:r51} converges to $0$, and Lemmas \ref{fluid} and \ref{waiting-tight} yields that the term in \eqref{eq:r52} converges to $0$. The lemma follows. 
\end{proof}

\begin{proof}[Proof of Theorem \ref{relation-thm}] Recall that $Q^{n}(0)\ge 0$. We then note that for $t\ge 0$, from \eqref{vw-fluid} and \eqref{total-abandon},  
\begin{align*}
\hatw^n_1(t) & = \left[ \frac{n}{\lambda^n_{-1}} \hatq^{n,+}(0) + \frac{\sqrt{n}}{\lambda^n_{-1}}  + \hatm^n_{1,1} + \hatm^n_{1,2} - \frac{n}{\lambda^n_{-1}}  \hatr^n_1 + \frac{n}{\lambda^n_{-1}}  \hatr^n_{-1}\right]^+, \\
\hatw^n_{-1}(t) & = \left[ -\frac{n}{\lambda^n_{1}} \hatq^{n,+}(0) + \frac{\sqrt{n}}{\lambda^n_{1}} + \hatm^n_{-1,1} + \hatm^n_{-1,2} - \frac{n}{\lambda^n_{1}}  \hatr^n_{-1} + \frac{n}{\lambda^n_{1}}  \hatr^n_{1}\right]^+,
\end{align*}
and from \eqref{queue-length},
\begin{align*}
\frac{\hatq^{n,+}(t)}{\lambda} & = \left[\frac{\hatq^n(0)}{\lambda} + \frac{\hatn^n_{1}(t)}{\lambda} -  \frac{\hatn^n_{-1}(t)}{\lambda} + \frac{(\lambda_1^n - \lambda_{-1}^n)t}{\lambda \sqrt{n}} -  \frac{\hatg^n_1(t)}{\lambda} +  \frac{\hatg^n_{-1}(t)}{\lambda}\right]^+,\\
\frac{\hatq^{n,-}(t)}{\lambda} & = \left[\frac{\hatq^n(0)}{\lambda} + \frac{\hatn^n_{1}(t)}{\lambda} -  \frac{\hatn^n_{-1}(t)}{\lambda} + \frac{(\lambda_1^n - \lambda_{-1}^n)t}{\lambda \sqrt{n}} -  \frac{\hatg^n_1(t)}{\lambda} +  \frac{\hatg^n_{-1}(t)}{\lambda}\right]^-.
\end{align*}
Recalling from Lemma \ref{martingale-conv-1}, we have 
\begin{equation}\label{eq:thm1}
\begin{aligned}
\left(\hatm^n_{1,1}, \hatm^n_{1,2}, \hatm^n_{-1,1}, \hatm^n_{-1,2}, \frac{\hatn^n_1}{\lambda}, \frac{\hatn^n_{-1}}{\lambda}\right) & \Go \left(0, \frac{c}{\lambda}, 0, -\frac{c}{\lambda}, 0, 0\right)\iota \\
& \quad + (Y_{-1}, - Y_1, Y_1, - Y_{-1}, -Y_1, -Y_{-1}), 
\end{aligned}
\end{equation}
where $\iota: [0,\infty)\to [0,\infty)$ is the identity map, and $Y_1(t) = \sigma_1 B_1(\lambda t)$ and $Y_{-1}(t) = \sigma_{-1} B_{-1}(\lambda t)$ with $B_1$ and $B_{-1}$ being two independent standard Brownian motions.
Then for $t\ge 0$,
\begin{align}
&\left|\hatw^n_1(t) - \frac{\hatq^{n,+}(t)}{\lambda} \right| + \left|\hatw^n_{-1}(t) - \frac{\hatq^{n,-}(t)}{\lambda} \right| \label{queue-wait-diff}\\
& \le \left| \left(\frac{n}{\lambda^n_{-1}} -\frac{1}{\lambda} \right)\hatq^{n}(0) + \frac{\sqrt{n}}{\lambda^n_{-1}}  +  \left(\hatm^n_{1,1} +\frac{\hatn^n_{-1}(t)}{\lambda} \right) + \left( \hatm^n_{1,2} - \frac{\hatn^n_{1}(t)}{\lambda} - \frac{(\lambda_1^n - \lambda_{-1}^n)t}{\lambda \sqrt{n}}\right) \right.  \nonumber\\
& \quad \left. + \left(- \frac{n}{\lambda^n_{-1}}  \hatr^n_1 +   \frac{\hatg^n_1(t)}{\lambda} \right) + \left( \frac{n}{\lambda^n_{-1}}  \hatr^n_{-1}  -  \frac{\hatg^n_{-1}(t)}{\lambda} \right) \right|  \nonumber\\
& \quad + \left| - \left(\frac{n}{\lambda^n_{1}} -\frac{1}{\lambda} \right)\hatq^{n}(0) + \frac{\sqrt{n}}{\lambda^n_{1}}  +  \left(\hatm^n_{-1,1} +\frac{\hatn^n_{1}(t)}{\lambda} \right) + \left( \hatm^n_{-1,2} - \frac{\hatn^n_{-1}(t)}{\lambda} + \frac{(\lambda_1^n - \lambda_{-1}^n)t}{\lambda \sqrt{n}}\right) \right.  \nonumber\\
& \quad \left. + \left(- \frac{n}{\lambda^n_{1}}  \hatr^n_{-1} +   \frac{\hatg^n_{-1}(t)}{\lambda} \right) + \left( \frac{n}{\lambda^n_{1}}  \hatr^n_{1}  -  \frac{\hatg^n_{1}(t)}{\lambda} \right) \right|. \nonumber
\end{align}
Combining \eqref{eq:thm1} and Lemma \ref{abandon-equal}, we have \eqref{queue-wait-diff} converges weakly to $0$.
\end{proof}

\begin{proof}[Proof of Theorem \ref{diffusion-thm}] From Theorem \ref{relation-thm}, 
\bes
\hatq^n - \lambda(\hatw^n_1 - \hatw^n_{-1}) \Go 0.
\ees
Let $(W_1, W_{-1})$ denote the weak limit of $(\hatw^n_1, \hatw^n_{-1})$ and $X$ be defined in Theorem \ref{waiting-convergence}. Then from Theorem \ref{waiting-convergence}, we see that 
\be\label{queue-limit}
\hatq^n \Go \lambda(\Psi_1(X) - \Psi_{-1}(X)) = \lambda X(t) - \lambda \int_0^\cdot H_1\left( W_1(s) \right) ds + \lambda \int_0^\cdot H_{-1}(W_{-1}(s))ds,
\ee 
Denote by $Q$ the weak limit of $\hatq^n$ in \eqref{queue-limit}. Then from Theorem \ref{relation-thm}, we have 
\bes
Q^+ = \lambda W_1, \ \mbox{and}  \ \ Q^- = \lambda W_{-1}. 
\ees
The result follows.
\end{proof}

\begin{proof}[Proof of Theorem \ref{stability}]
We first obtain the generator of $Q$ as follows. For $x\in \RR$ and $f\in C^2_0(\RR)$,
\[
\mathcal{A}f(x) = \frac{\lambda^3(\sigma_1^2+\sigma_{-1}^2) f''(x)}{2} + \left[c - \lambda H_1\left(\frac{x}{\lambda}\right) 1_{\{x\ge 0\}}  + \lambda H_{-1}\left(\frac{x}{\lambda}\right) 1_{\{x<0\}} \right] f'(x).
\]
From Proposition 9.2 in \cite{ek86}, it suffices to verify that $\int_\RR \mathcal{A} f(x)\pi(dx) =0$ for all $f\in C^2_0(\RR)$. Indeed, using integration by parts, we have 
\begin{align*} 
& \int_{[0,\infty)} \mathcal{A} f(x)\pi(dx) \\
& = C_0 \int_{[0,\infty)}\frac{\lambda^3(\sigma_1^2+\sigma_{-1}^2) f''(x)}{2} \exp\left\{- \frac{2}{\lambda^3(\sigma^2_1+\sigma^2_{-1})} \left(-cx + \lambda^2 \int_0^{\frac{x}{\lambda}} H_1(u) \D u  \right) \right\} \D x\\
& \quad + C_0 \int_{[0,\infty)}\left[c - \lambda H_1\left(\frac{x}{\lambda}\right)  \right] f'(x) \exp\left\{- \frac{2}{\lambda^3(\sigma^2_1+\sigma^2_{-1})} \left(-cx + \lambda^2 \int_0^{\frac{x}{\lambda}} H_1(u) \D u  \right) \right\} \D x \\
& = - \frac{\lambda^3(\sigma_1^2+\sigma_{-1}^2) f'(0)}{2} C_0,
\end{align*}
and 
\begin{align*} 
& \int_{(-\infty, 0)} \mathcal{A} f(x)\pi(dx) \\
& = C_0 \int_{(-\infty, 0)}\frac{\lambda^3(\sigma_1^2+\sigma_{-1}^2) f''(x)}{2} \exp\left\{- \frac{2}{\lambda^3(\sigma^2_1+\sigma^2_{-1})} \left(-cx + \lambda^2 \int_0^{-\frac{x}{\lambda}} H_{-1}(u) \D u  \right) \right\} \D x\\
& \quad + C_0 \int_{(-\infty,0)}\left[c + \lambda H_{-1}\left(\frac{x}{\lambda}\right)  \right] f'(x) \exp\left\{- \frac{2}{\lambda^3(\sigma^2_1+\sigma^2_{-1})} \left(-cx + \lambda^2 \int_0^{-\frac{x}{\lambda}} H_{-1}(u) \D u  \right) \right\} \D x \\
& = \frac{\lambda^3(\sigma_1^2+\sigma_{-1}^2) f'(0)}{2} C_0.
\end{align*}
Thus $\pi$ is a stationary distribution of $Q$. Finally, the uniqueness of $\pi$ follows from the irreducibility of $Q$. 

\end{proof}

\subsection{Proof of Corollary \ref{cor} }\label{coro:proof}
The goal is to prove $\hatq^n$ is $C$-tight under Assumption \ref{arrival}, \ref{htc}, and \ref{hazard}. We first follow \cite{kang10} to construct the compensator for the abandonment process $G^n_i$. Define the potential waiting time process for the customers of Class $i$ as follows: For $k\in \NN,$
\bes
\tilde w^n_{i,k}(t) = \begin{cases} \max\{t - t^n_{i,k}, 0\}, & \mbox{if} \ t - t^n_{i,k} < d^n_{i,k}, \\
d^n_{i,k}, & \mbox{otherwise}.
\end{cases}
\ees
We note that for $k\in\NN$, $\tilde w^n_{i,k}(t)$ represents the amount of time spent by the $k^{th}$ customer of Class $i$ since entering the system, and remains constant at $d^n_{i,k}$ once the time spent reaches the patience time, and for $k\in -\NN\cup\{0\}$, $\tilde w^n_{i,k}$ represents the potential waiting time process of the $(-k+1)^{st}$ customer who enters the system before time $0$ (if such customer exists). The following measure, which is called the potential queue measure, assigns a unit mass to the potential waiting time of each customer of Class $i$ that has entered the system by time $t$ and whose potential waiting time has not yet reached the patience time. 
\bes
\eta^n_{i,t}(\D x) = \sum_{k=-Q^{n}_i(0)+1}^{N^n_i(t)} \delta_{\tilde w^n_{i,k}}(\D x) 1_{\{\tilde w^n_{i,k}(t)<d^n_{i,k}\}} = \sum_{k=-Q^{n}_i(0)+1}^{N^n_i(t)} \delta_{\tilde w^n_{i,k}}(\D x) 1_{\left\{\frac{\D \tilde w^n_{i,k}(t+)}{\D t}>0\right\}}.
\ees
The number of Class $i$ customers at time $t$ can then be formulated as follows:
\bes
Q^{n}_i(t) =  \sum_{k=-Q^{n}_i(0)+1}^{N^n_i(t)} 1_{\{\tilde w^n_{i,k}(t) \le \chi^n_i(t), \ \tilde w^n_{i,k}(t)<d^n_{i,k}\}} = \eta^n_{i,t}[0,\chi^n_i(t)],
\ees
where $\chi^n_i(t) = \inf\{x>0: \eta^n_{i,t}[0,x] \ge Q^{n}_i(t)\}$ which is the waiting time of the head-of-the-line Class $i$ customer in the queue at time $t.$ The abandonment process of Class $i$ at time $t$ becomes
\bes
G^n_i(t) = \sum_{k=-Q^{n}_i(0)+1}^{N^n_i(t)} \sum_{s\in[0,t]} 1_{\left\{\tilde w^n_{i,k}(s) \le \chi^n_i(s-), \ \frac{\D \tilde w^n_{i,k}(s-)}{\D t}>0, \ \frac{\D \tilde w^n_{i,k}(s+)}{\D t}=0\right\}}.
\ees
Recall from Assumption \ref{hazard} that for $i=\pm 1$, the patience time distribution function $F^n_i$ has density function $f^n_i$ and hazard rate function $h^n_i  = f^n_i/(1-F^n_i)$ on $[0, M^n_i)$, where $M^n_i = \sup\{s\ge 0: F^n(s) < 1\}$. Furthermore, for $t\in [0, M^n_i)$, $\sup_{n\in\NN}\|h^n_i\|_t < \infty$, and $h^n_i(t/\sqrt{n})\to h_i(t)$, where $h_i$ is a nonnegative measurable function. 
Define for $t\ge 0,$
\bes
A^n_i(t) = \int_0^t \left( \int_{[0,M_i^n)} 1_{[0,\chi^n_i(s-)]}(x) h_i^n(x) \eta^n_{i,s}(\D x)\right) \D s, 
\ees
and
\begin{align*}
\clg^n_{i,t} & = \sigma\left\{Q^n(0); (N^n_i(s), N^n_{-i}(s)), 0\le s \le t;  \right. \nonumber\\
\begin{split}
& \quad\quad\quad  (w^n_{i,l}, d^n_{i,l}), l =-Q^n_i(0)+1, \ldots, 0, 1,\ldots, N_i(t); \\
& \quad\quad\quad  (w^n_{-i,l}, d^n_{-i,l}), l =-Q^n_{-i}(0)+1, \ldots, 0,  1, \ldots, N^n_{-i}(t);
\end{split}\label{filtration}\\
& \quad\quad\quad \left. t^n_{-i,l}, l = 1, 2, \ldots, N^n_i(t)+Q_i^n(0) -R_i^n(t) - Q^n_{-i}(0) + R^n_{-i}(t) \right\}. \nonumber
\end{align*}

\begin{lemma}\label{martingale}
For $n\in\NN,$ the process $A^n_i$ is the $\{\clg^n_{i,t}\}$ compensator of the process $G^n_i$, and in particular, $G^n_i - A^n_i$ is a $\{\clg^n_{i,t}\}$ local  martingale. 
\end{lemma}
\begin{proof}
First it can be seen that $(Q^n, G^n_i, A^n_i)\in \{\clg^n_{i,t}\}$. Then the rest of the proof is essentially the same as that of Lemma 5.4 in \cite{kang10}.\end{proof}

\begin{lemma}\label{compensator}
For $t\ge 0,$ 
\bes
A^n_i(t) =   \sum_{k=-Q^n_i(0)+1}^{N^n_i(t)} \int_0^{(t-t_{i,k}^n)\wedge w^n_{i,k} \wedge d^n_{i,k}} h^n_i(u) \D u.
\ees
\end{lemma}
\begin{proof} The following proof is adapted from \cite{ReedTezcan12}. We first observe that for $t\ge 0$, 
\bes
1_{[0,\chi^n_i(s-)]}(x) \eta^n_{i,s}(dx) =  \sum_{k=-Q^n_i(0)+1}^{N^n_i(s)}  1_{\{s-t^n_{i,k}<w^n_{i,k}\}}(x) 1_{\{s-t^n_{i,k}<d^n_{i,k}\}}(x) \delta_{s-t^n_{i,k}}(\D x),
\ees
and so
\bes
 \int_{[0,M_i^n)} 1_{[0,\chi^n_i(s-)]}(x) h_i^n(x) \eta^n_{i,s}(dx) = \sum_{k=-Q^n_i(0)+1}^{N^n_i(s)}   h^n_i(s-t^n_{i,k}) 1_{\{s-t^n_{i,k}<w^n_{i,k}\wedge d^n_{i,k}\}}. 
\ees
Finally, we have for $t\ge 0,$
\begin{align*}
A^n_i(t) & = \int_0^t \sum_{k=-Q^n_i(0)+1}^{N^n_i(s)}   h^n_i(s-t^n_{i,k}) 1_{\{s-t^n_{i,k}<w^n_{i,k}\wedge d^n_{i,k}\}} \D s \\
& = \sum_{k=-Q^n_i(0)+1}^{N^n_i(t)} \int_{t^n_{i,k}}^t   h^n_i(s-t^n_{i,k}) 1_{\{s-t^n_{i,k}<w^n_{i,k}\wedge d^n_{i,k}\}} \D s \\
& = \sum_{k=-Q^n_i(0)+1}^{N^n_i(t)} \int_0^{(t-t^n_{i,k})\wedge w^n_{i,k}\wedge d^n_{i,k}} h^n_i(u)\D u.
\end{align*}
\end{proof}

\begin{lemma}\label{queue-tightness}
Assume $\hatq^n(0)$ converges weakly to some random variable $q$. Then $\hatq^n$ is $C$-tight. 
\end{lemma}
\begin{proof}
For $t\ge 0$, we note that 
\begin{align*}
\hatq^n(t) & = \hatq^n(0) + \hatn^n_1(t) - \hatn^n_{-1}(t) + \frac{(\lambda^n_1 - \lambda^n_{-1})t}{\sqrt{n}}- \hatg^n_1(t) + \hatg^n_{-1}(t).
\end{align*}
Define for $t\ge 0$, 
\begin{align*}
\hata^n_i(t) = \frac{1}{\sqrt{n}} A^n_i(t), \ \mbox{and} \ \hatm^n_i(t) = \hatg^n_i(t) - \hata^n_i(t).
\end{align*}
From Lemmas \ref{martingale} and \ref{compensator}, the quadratic variation of the local martingale $\hatm^n_i$ is 
\be\label{qv}\ba
[] [\hatm^n_i]_t & = \frac{1}{n} A^n_i(t) = \frac{1}{n}\int_0^t \sum_{k=-Q^n_i(0)+1}^{N^n_i(s)}   h^n_i(s-t^n_{i,k}) 1_{\{s-t^n_{i,k}<w^n_{i,k}\wedge d^n_{i,k}\}} \D s \\
& \le  \frac{\|h^n_i\|_t}{n}\int_0^t \sum_{k=-Q^n_i(0)+1}^{N^n_i(s)}   1_{\{s-t^n_{i,k}<w^n_{i,k}\wedge d^n_{i,k}\}} \D s \\
& =  \frac{\|h^n_i\|_t}{n}\int_0^t Q^n_i(s) \D s = \|h^n_i\|_t \int_0^t \barq^n_i(s) \D s  \to 0, \ \mbox{in probability.}
\ea\ee
Thus from Lemma \ref{martingale-convergence}, we have $\hatm^n_i \Go 0$. Furthermore, we have for $t\ge 0$,
\begin{align*}
\|\hatq^n\|_t & \le |\hatq^n(0)| + \sum_{i=\pm 1}( \|\hatn^n_i\|_t + \|\hatm^n_i\|_t)+ \frac{|\lambda^n_1 - \lambda^n_{-1}|t}{\sqrt{n}}  + \sum_{i=\pm 1}\|h^n_i\|_t  \int_0^t \|\hatq^n\|_s \D s.
\end{align*}
Using Gronwall's inequality, we have 
\be\label{boundedness}
\|\hatq^n\|_t \le e^{\sum_{i=\pm1}\|h^n_i\|_t} \left( |\hatq^n(0)| + \sum_{i=\pm 1}( \|\hatn^n_i\|_t + \|\hatm^n_i\|_t) + \frac{|\lambda^n_1 - \lambda^n_{-1}|t}{\sqrt{n}}\right).
\ee
Next using similar argument as in \eqref{qv}, we have for $0\le s \le t < \infty,$
\begin{align*}
\hata^n_i(t) - \hata^n_i(s) & = \frac{1}{\sqrt{n}}\int_s^t \sum_{k=-Q^n_i(0)+1}^{N^n_i(u)}   h^n_i(u-t^n_{i,k}) 1_{\{u-t^n_{i,k}<w^n_{i,k}\wedge d^n_{i,k}\}} \D u \\
& \le \| h^n_i\|_{s,t} \int_s^t \hatq^n_i(u) \D u,
\end{align*}
and therefore,
\begin{align*}
& |\hatq^n(t) - \hatq^n(s)| \\
& \le \sum_{i=\pm 1} (|\hatn^n_i(t) - \hatn^n_i(s)| + |\hatm^n_{i}(t) - \hatm^n_{i}(s)| + |\hata^n_i(t) - \hata^n_i(s)|) + \frac{(\lambda^n_1-\lambda^n_{-1})(t-s)}{\sqrt{n}}\\
& \le  \sum_{i=\pm 1} (|\hatn^n_i(t) - \hatn^n_i(s)| + |\hatm^n_{i}(t) - \hatm^n_{i}(s)|) + \sum_{i=\pm1}\| h^n_i\|_{s,t} \int_s^t |\hatq^n(u)| \D u + \frac{(\lambda^n_1-\lambda^n_{-1})(t-s)}{\sqrt{n}}\\
& \le \sum_{i=\pm 1} (|\hatn^n_i(t) - \hatn^n_i(s)| + |\hatm^n_{i}(t) - \hatm^n_{i}(s)|) + (t-s)\|\hatq^n\|_t \sum_{i=\pm1}\| h^n_i\|_{s,t} + \frac{(\lambda^n_1-\lambda^n_{-1})(t-s)}{\sqrt{n}}.
\end{align*}
The above estimate and \eqref{boundedness} implies that $\hatq^n$ is $C$-tight on noting that $\hatn^n_i$ and $\hatm^n_i$ are weakly convergent, and that the jump seize of $\hatq^n_i$ is bounded by $\frac{1}{\sqrt{n}}$.
\end{proof}

\begin{proof}[Proof of Corollary \ref{cor}] This follows immediately from Lemma \ref{queue-tightness}, Theorem \ref{diffusion-thm}, and Remark \ref{remark} (iv).
\end{proof}

\section*{Appendix}

\begin{proof}[Proof of Proposition \ref{unique-map}] Fix $T>0$, and let $\kappa >0$ be such that for $t_1, t_2\in [0, T]$, 
\begin{align*}
|H_1(t_1) - H_1(t_2)| \le \kappa |t_1 -t_2|, \ \mbox{and} \ |H_{-1}(t_1) - H_{-1}(t_2)| \le \kappa |t_1 -t_2|.
\end{align*}
(i) Let $w(t) = w_1(t) + w_{-1}(t), t\in [0, T]$. We first show that there exists $M\in (0,\infty)$ such that $\|w\|_T \le M.$ We first note that for $t\ge 0,$
\begin{align*}
w(t) & = \left|x(t) - \int_0^t H_1(w_1(s))\D s + \int_0^t H_{-1}(w_{-1}(s))\D s\right| \\
& \le |x(t)| + \int_0^t (H_1+H_{-1})(w(s)) \D s.  
\end{align*}
Let $H(x) = H_1(x)+H_{-1}(x)+1, x\in [0,\infty).$ Then we have for $t\in [0,T],$
\begin{align}\label{eq:54}
\|w\|_t \le \|x\|_t + \int_0^t H(\|w\|_s) \D s.
\end{align}
Let $y(t) = \int_0^t H(\|w\|_s) ds, t\in[0, T]$. Then from \eqref{eq:54}, for $t\in[0,T]$,
\[
y'(t) = H(\|w\|_t) \le H(\|x\|_T+y(t)), 
\]
Integrating the above equation on $[0, T],$ we have
\begin{equation}
\int_0^T \frac{y'(t)}{H(\|x\|_T+y(t))} dt \le T,
\end{equation}
and so
\be\label{eq:55}
\int_{\|x\|_T}^{\|x\|_T+y(T)} \frac{du}{H(u)} \le T.
\ee
Now define for $t\in [0,T],$
\[
\Phi(t) = \int_0^t \frac{du}{H(u)}.
\]
Since $H$ is strictly positive, nondecreasing, and Lipschitz continous on $[0,T]$, it is clear that $\Phi$ is increasing and differentiable. Finally, we have from \eqref{eq:55}, 
\begin{align*}
\Phi(\|x\|_T + y(T)) - \Phi(\|x\|_T) \le T,
\end{align*}
and so
\begin{align*}
y(T) \le \Phi^{-1}(\Phi(\|x\|_T) + T) - \|x\|_T.
\end{align*}
Now from \eqref{eq:54},
\begin{align}\label{eq:56}
\|w\|_T \le \Phi^{-1}(\Phi(\|x\|_T) + T) .
\end{align}
Denoting the RHS of \eqref{eq:56} by $M$, we have $\|w\|_T \le M.$ We next use Picard iteration to show the existence of $(w_1, w_{-1})$. Define a sequence of $(w^n_1, w^n_{-1})_{n=0}^\infty$ recursively as follows. Let $(w^0_1, w^0_{-1}) = 0$, and for $n\ge 1$ and $t\in [0,T],$
\begin{align*}
w_1^n(t) & = \left[x(t) - \int_0^t H_1(w_1^{n-1}(s))\D s + \int_0^t H_{-1}(w_{-1}^{n-1}(s))\D s\right]^+, \\
w_{-1}^n(t) & = \left[x(t) - \int_0^t H_1(w_1^{n-1}(s))\D s + \int_0^t H_{-1}(w_{-1}^{n-1}(s))\D s\right]^-.
\end{align*}
Then for $s\in [0,T],$
\be
\|w^1_1\|_s  \le \|x\|_s, \ \ \|w^1_{-1}\|_s  \le \|x\|_s,
\ee
for $n\ge 1$,
\begin{align}
& \|w^n_1 - w^{n-1}_1\|_s \nonumber \\
& = \sup_{0\le t\le s} \left| \left[x(t) - \int_0^t H_1(w_1^{n-1}(s))\D s + \int_0^t H_{-1}(w_{-1}^{n-1}(s))\D s\right]^+ \right. \nonumber\\
& \quad\quad\quad\quad - \left. \left[x(t) - \int_0^t H_1(w_1^{n-2}(s))\D s + \int_0^t H_{-1}(w_{-1}^{n-2}(s))\D s\right]^+\right| \nonumber\\
& \le \sup_{0\le t\le s} \left|   \int_0^t - H_1(w_1^{n-1}(s)) + H_1(w_1^{n-2}(s))\D s + \int_0^t H_{-1}(w_{-1}^{n-1}(s))- H_{-1}(w_{-1}^{n-2}(s))\D s\right| \nonumber\\
& \le \int_0^s | H_1(w_1^{n-1}(s)) - H_1(w_1^{n-2}(s))| \D s + \int_0^s |H_{-1}(w_{-1}^{n-1}(s))- H_{-1}(w_{-1}^{n-2}(s))| \D s \label{eq:57}\\
& \le \kappa s (\|w^{n-1}_1 - w^{n-2}_1\|_s + \|w^{n-1}_{-1} - w^{n-2}_{-1}\|_s), \nonumber
\end{align}
and using similar analysis, we have for $n\ge 1$ and $s\in [0,T],$
\bes
\|w^n_{-1} - w^{n-1}_{-1}\|_s \le \kappa s (\|w^{n-1}_1 - w^{n-2}_1\|_s + \|w^{n-1}_{-1} - w^{n-2}_{-1}\|_s).
\ees
Thus for $s\in[0,T],$
\be\label{eq:59}
\|w^1_1\|_s + \|w^1_{-1}\|_s \le 2\|x\|_s,
 \ee 
 and
 \be
 \|w^n_1 - w^{n-1}_1\|_s + \|w^n_{-1} - w^{n-1}_{-1}\|_s \le 2\kappa s (\|w^{n-1}_1 - w^{n-2}_1\|_s + \|w^{n-1}_{-1} - w^{n-2}_{-1}\|_s), \ \ n\ge 2.
 \ee
The rest of the proof is similar to the proof of Lemma 1 in \cite{Reed08}. Choose $\delta >0$ such that $2\kappa\delta < 1$, and then partition $[0, T]$ into $\lfloor T/\delta \rfloor+1$ subintervals $[y_j, y_{j+1}], j =0, 1, \ldots, \lfloor T/\delta \rfloor$, where $y_0 =0, y_j = j\delta, j=1, \ldots, \lfloor T/\delta \rfloor$, and $y_{\lfloor T/\delta \rfloor+1} = T$. Let $c = 2 \|x\|_T$. We first observe that for $n\ge 1$,
\bes
\|w^n_1 - w^{n-1}_1\|_\delta + \|w^n_{-1} - w^{n-1}_{-1}\|_\delta \le 2\kappa \delta (\|w^{n-1}_1 - w^{n-2}_1\|_\delta + \|w^{n-1}_{-1} - w^{n-2}_{-1}\|_\delta),
\ees
and repeating the iterations, we have
\be\label{eq:511}
\|w^n_1 - w^{n-1}_1\|_\delta + \|w^n_{-1} - w^{n-1}_{-1}\|_\delta \le (2\kappa \delta)^{n-1} \cdot 2\|x\|_\delta \le (2\kappa \delta)^{n-1}  c.
\ee
In the following, we use induction to show that for $j =1, \ldots, \lfloor T/\delta \rfloor+1,$
\be\label{eq:512}
\|w^n_1 - w^{n-1}_1\|_{y_j} + \|w^n_{-1} - w^{n-1}_{-1}\|_{y_j} \le j (n-1)^j(2\kappa \delta)^{n-1} c.
\ee
From \eqref{eq:511}, the above inequality \eqref{eq:512} holds for $j=1.$ Suppose \eqref{eq:512} holds for $j=1, \ldots, k.$ Now from \eqref{eq:57},
\begin{align*}
& \|w^n_1 - w^{n-1}_1\|_{y_{k+1}} + \|w^n_{-1} - w^{n-1}_{-1}\|_{y_{k+1}} \\
& \le 2\sum_{j=1}^k \int_{y_{j-1}}^{y_j} | H_1(w_1^{n-1}(s)) - H_1(w_1^{n-2}(s))| + |H_{-1}(w_{-1}^{n-1}(s))- H_{-1}(w_{-1}^{n-2}(s))| \D s \\
& \quad + 2\int_{y_k}^{y_{k+1}} | H_1(w_1^{n-1}(s)) - H_1(w_1^{n-2}(s))| + |H_{-1}(w_{-1}^{n-1}(s))- H_{-1}(w_{-1}^{n-2}(s))| \D s \\
& \le 2\kappa \delta\sum_{j=1}^k  (\|w^{n-1}_1 - w^{n-2}_1\|_{y_j}+\|w^{n-1}_{-1} - w^{n-2}_{-1}\|_{y_j}) +2\kappa \delta (\|w^{n-1}_1 - w^{n-2}_1\|_{y_{k+1}}+\|w^{n-1}_{-1} - w^{n-2}_{-1}\|_{y_{k+1}}) \\
& \le 2\kappa \delta \sum_{j=1}^kj (n-2)^j(2\kappa \delta)^{n-2} c +2\kappa \delta (\|w^{n-1}_1 - w^{n-2}_1\|_{y_{k+1}}+\|w^{n-1}_{-1} - w^{n-2}_{-1}\|_{y_{k+1}}).
\end{align*}
There exists $n_0 \in \NN$ such that when $n\ge n_0$, 
\[
\sum_{j=1}^kj (n-2)^j \le k (n-1)^k. 
\]
Thus when $n\ge n_0$,
\begin{align}
& \|w^n_1 - w^{n-1}_1\|_{y_{k+1}} + \|w^n_{-1} - w^{n-1}_{-1}\|_{y_{k+1}} \nonumber \\
& \le c k(n-1)^k (2\kappa \delta)^{n-1}  +2\kappa \delta (\|w^{n-1}_1 - w^{n-2}_1\|_{y_{k+1}}+\|w^{n-1}_{-1} - w^{n-2}_{-1}\|_{y_{k+1}}). \label{eq:513}
\end{align}
From \eqref{eq:513} and \eqref{eq:59}, we have 
\begin{align*}
& \|w^2_1 - w^{1}_1\|_{y_{k+1}} + \|w^2_{-1} - w^{1}_{-1}\|_{y_{k+1}} \\
& \le c k \cdot 2\kappa \delta +2\kappa \delta (\|w^{1}_1\|_{y_{k+1}}+\|w^{1}_{-1}\|_{y_{k+1}}) \\
& \le c(k+1) \cdot 2\kappa \delta.
\end{align*}
Iterating \eqref{eq:513} shows that for large enough $n$,
\begin{align*}
& \|w^n_1 - w^{n-1}_1\|_{y_{k+1}} + \|w^n_{-1} - w^{n-1}_{-1}\|_{y_{k+1}} \\
& \le c (2\kappa \delta)^{n-1}\left( k\sum_{j=1}^{n-1}j^k +1 \right) \le  c (2\kappa \delta)^{n-1} (k (n-1)^{k+1} +1) \\
& \le c (2\kappa \delta)^{n-1} (k+1) (n-1)^{k+1}.
\end{align*}
Thus we show \eqref{eq:512} holds. When $j= \lfloor T/\delta \rfloor + 1$, 
\begin{align*}
& \|w^n_1 - w^{n-1}_1\|_{T} + \|w^n_{-1} - w^{n-1}_{-1}\|_{T}  \le   c (2\kappa \delta)^{n-1} (\lfloor T/\delta \rfloor+2) (n-1)^{\lfloor T/\delta \rfloor+2} \to 0, \ \ \mbox{as $n\to\infty$.}
\end{align*}
So $(w^n_1, w^n_{-1})$ converges and denote the limit function by $(w^*_1, w^*_{-1}) = \lim_{n\to\infty}(w^n_1, w^n_{-1}).$ From \eqref{eq:51} and \eqref{eq:52}, define the functional $G = (G_1, G_{-1}):D([0,T], \RR^2_+) \to D([0, T], \RR^2_+)$ as follows: For $t\in [0, T],$
\begin{align*}
G(w)_1(t) & = \left[x(t) - \int_0^t H_1(w_1(s))\D s + \int_0^t H_{-1}(w_{-1}(s))\D s\right]^+, \\
G(w)_{-1}(t) & = \left[x(t) - \int_0^t H_1(w_1(s))\D s + \int_0^t H_{-1}(w_{-1}(s))\D s\right]^-.
\end{align*}
Thus $w = G(w)$. We next note that $G$ is Lipschitz continuous, i.e., for $s\in [0,T]$
\[
\|G(w)- G(\tilde w)\|_s \le 2 \kappa T \|w-\tilde w\|_s.
\]
Using Banach fixed point theorem, we have $w = (w^*_1, w^*_{-1}).$\\
(ii) Let $x^n, x\in D([0,\infty), \RR)$ such that $x^n \to x$ in the Skorohod $J_1$ topology. Fix $T\in (0,\infty).$ Then there exists a sequence of strictly increasing homeomorphism $\{\lambda^n\}$ on $[0,\infty)$ such that
\be\label{eq:53}
\sup_{0\le t\le T}|x^n(\lambda^n(t)) - x(t)| \vee \sup_{0\le t\le T}|\lambda^n(t) - t| \to 0, \ \mbox{as $n\to\infty.$}
 \ee
Let 
\[
w^n \equiv (w^n_1, w^n_{-1}) = (\Psi_1(x^n), \Psi_{-1}(x^n)), \ \mbox{and} \ w\equiv (w_1, w_{-1}) = (\Psi_1(x), \Psi_{-1}(x)).
\]
It is clear that, from \eqref{eq:53}, there exists $M>0$ such that 
\[
\sup_{n\in\NN}\|x^n\|_T + \|x\|_T \le M,
\]
and from \eqref{eq:56}, there exists $\tilde M>0$ such that
\[
\sup_{n\in\NN}\|w^n\|_T + \|w\|_T \le \tilde M.
\]
Let $\kappa>0$ be the Lipschitz constant for $H_1$ and $H_{-1}$ on $[0, \tilde M].$
For $u\in[0, T]$, 
\begin{align*}
\sup_{0\le t \le u} |w^n_1(\lambda^n(t)) - w_1(t)| & = \sup_{0\le t \le u}\left|\left[x^n(\lambda^n(t)) - \int_0^{\lambda^n(t)} H_1(w_1^n(s))\D s + \int_0^{\lambda^n(t)} H_{-1}(w_{-1}^n(s))\D s\right]^+ \right. \\
& \quad\quad\quad - \left. \left[x(t) - \int_0^t H_1(w_1(s))\D s + \int_0^t H_{-1}(w_{-1}(s))\D s\right]^+ \right| \\
& \le \sup_{0\le t\le u}|x^n(\lambda^n(t))-x(t)| + \sup_{0\le t \le u}\left| \int_0^{\lambda^n(t)} H_1(w_1^n(s))\D s - \int_0^t H_1(w_1(s))\D s\right|\\
& \quad\quad\quad + \sup_{0\le t\le u} \left|\int_0^{\lambda^n(t)} H_{-1}(w_{-1}^n(s))\D s- \int_0^t H_{-1}(w_{-1}(s))\D s\right|.
\end{align*}
We next observe that for $u\in[0,T]$,
\begin{align*}
&\int_0^{\lambda^n(t)} H_1(w_1^n(s))\D s - \int_0^t H_1(w_1(s))\D s\\
& = \int_0^{\lambda^n(t)} H_1(w_1^n(s))\D s -\int_0^{t} H_1(w_1^n(\lambda^n(s)))\D s + \int_0^{t} H_1(w_1^n(\lambda^n(s)))\D s - \int_0^t H_1(w_1(s))\D s \\
& = \int_0^{t} H_1(w_1^n(s)) - H_1(w_1^n(\lambda^n(s)))\D s + \int_t^{\lambda^n(t)} H_1(w_1^n(s))\D s  + \int_0^t  H_1(w_1^n(\lambda^n(s))) - H_1(w_1(s)) \D s.
\end{align*}
We note that from the Lipschitz continuity of $H_1$ and \eqref{eq:51}, 
\begin{align*}
& \sup_{0\le t\le T}\left|\int_0^{t} H_1(w_1^n(s)) - H_1(w_1^n(\lambda^n(s)))\D s\right| \\
& \le T \kappa \sup_{0\le t\le T}\left|w_1^n(t) - w_1^n(\lambda^n(t))\right| \\
& \le T\kappa  \sup_{0\le t\le T} \left(\left| x^n(t) - x^n(\lambda^n(t))  \right|  + \left| \int_t^{\lambda^n(t)} H_1(w_1(s)) \D s\right| + \left|\int_t^{\lambda^n(t)} H_{-1}(w_{-1}(s)) \D s\right|\right) \\
& \le \kappa T \sup_{0\le t\le T}\left| x^n(t) - x^n(\lambda^n(t))  \right|  + 2\kappa^2 T \tilde M \sup_{0\le t\le T} |\lambda^n(t) -t|
\end{align*}
Thus
\begin{align*}
& \sup_{0\le t\le u} \left|\int_0^{\lambda^n(t)} H_1(w_1^n(s))\D s - \int_0^t H_1(w_1(s))\D s\right| \\
& \le \sup_{0\le t\le T}\left|\int_0^{t} H_1(w_1^n(s)) - H_1(w_1^n(\lambda^n(s)))\D s\right| + \sup_{0\le t\le T} \left| \int_t^{\lambda^n(t)} H_1(w_1^n(s))\D s\right| \\
& \quad + \sup_{0\le t\le T} \left| \int_0^t  H_1(w_1^n(\lambda^n(s))) - H_1(w_1(s)) \D s\right| \\
& \le \kappa T \sup_{0\le t\le T}\left| x^n(t) - x^n(\lambda^n(t))  \right|  + 2\kappa^2 T \tilde M \sup_{0\le t\le T} |\lambda^n(t) -t| + \kappa \tilde M \sup_{0\le t\le T} |\lambda^n(t) - t|  \\
& \quad + 2\kappa \int_0^T \sup_{0\le s \le t} |w^n_1(\lambda^n(s)) - w_1(s)| \D t \\
& = \kappa T \sup_{0\le t\le T}\left| x^n(t) - x^n(\lambda^n(t))  \right|  + (2\kappa^2+\kappa) T \tilde M \sup_{0\le t\le T} |\lambda^n(t) -t|  + 2\kappa \int_0^T \sup_{0\le s \le t} |w^n_1(\lambda^n(s)) - w_1(s)| \D t,
\end{align*}
and similarly, 
\begin{align*}
& \sup_{0\le t\le u} \left|\int_0^{\lambda^n(t)} H_{-1}(w_{-1}^n(s))\D s- \int_0^t H_{-1}(w_{-1}(s))\D s\right|  \\
&  \le  \kappa T \sup_{0\le t\le T}\left| x^n(t) - x^n(\lambda^n(t))  \right|  + (2\kappa^2+\kappa) T \tilde M \sup_{0\le t\le T} |\lambda^n(t) -t|  + 2\kappa \int_0^T \sup_{0\le s \le t} |w^n_1(\lambda^n(s)) - w_1(s)| \D t.
\end{align*}
Combining the above estimates, we have
\begin{align*}
&\sup_{0\le t \le T} |w^n_1(\lambda^n(t)) - w_1(t)| \le (1+2\kappa T) \sup_{0\le t\le T}|x^n(\lambda^n(t))-x(t)|\\
 & + 2 (2\kappa^2+\kappa) T \tilde M \sup_{0\le t\le T} |\lambda^n(t) - t| + 2\kappa  \int_0^T \sup_{0\le s \le t} |w^n_1(\lambda^n(s)) - w_1(s)| + \sup_{0\le s \le t}|w^n_{-1}(\lambda^n(s)) - w_{-1}(s)|\D t. 
\end{align*}
A similar analysis yields that 
\begin{align*}
&\sup_{0\le t \le T} |w^n_{-1}(\lambda^n(t)) - w_{-1}(t)| \le (1+2\kappa T)\sup_{0\le t\le T}|x^n(\lambda^n(t))-x(t)| \\
 & + 2 (2\kappa^2+\kappa) T \tilde M \sup_{0\le t\le T} |\lambda^n(t) - t| + 2\kappa  \int_0^T \sup_{0\le s \le t} |w^n_1(\lambda^n(s)) - w_1(s)| + \sup_{0\le s \le t}|w^n_{-1}(\lambda^n(s)) - w_{-1}(s)|\D t. 
\end{align*} 
Finally, we have for $u\in[0,T],$
\begin{align*}
&\sup_{0\le t \le u} |w^n_1(\lambda^n(t)) - w_1(t)| + \sup_{0\le t \le u} |w^n_{-1}(\lambda^n(t)) - w_{-1}(t)|\le  2(1+2\kappa T)\sup_{0\le t\le u}|x^n(\lambda^n(t))-x(t)| \\
 & + 4 (2\kappa^2+\kappa) T \tilde M\sup_{0\le t\le u} |\lambda^n(t) - t| + 4\kappa  \int_0^T \sup_{0\le s \le t} |w^n_1(\lambda^n(s)) - w_1(s)| + \sup_{0\le s \le t}|w^n_{-1}(\lambda^n(s)) - w_{-1}(s)|\D t. 
\end{align*} 
Using Gronwall's inequality, we have for $u\in[0,T],$
\begin{align*}
& \sup_{0\le t \le T} |w^n_1(\lambda^n(t)) - w_1(t)| + \sup_{0\le t \le u} |w^n_{-1}(\lambda^n(t)) - w_{-1}(t)|\\
& \le \left( 2(1+2\kappa T)\sup_{0\le t\le T}|x^n(\lambda^n(t))-x(t)| + 4 (2\kappa^2+\kappa) T \tilde M \sup_{0\le t\le T} |\lambda^n(t) - t|\right) e^{4\kappa u}.
\end{align*}
Finally, from the above estimate, we see that 
\begin{align*}
\sup_{0\le t \le T} |w^n(\lambda^n(t)) - w(t)| \vee \sup_{0\le t\le T}|\lambda^n(t) - t| \to 0, \ \mbox{as $n\to\infty.$} 
\end{align*}
\end{proof}

 \bibliographystyle{amsplain}
 \bibliography{references}

\providecommand{\bysame}{\leavevmode\hbox to3em{\hrulefill}\thinspace}
\providecommand{\MR}{\relax\ifhmode\unskip\space\fi MR }
\providecommand{\MRhref}[2]{%
  \href{http://www.ams.org/mathscinet-getitem?mr=#1}{#2}
}
\providecommand{\href}[2]{#2}
\begin{thebibliography}{10}

\bibitem{adm14}
P.~Afeche, A.~Diamant, and J.~Milner, \emph{Double-sided batch queues with
  abandonment: {M}odeling crossing networks}, Operations Research \textbf{62}
  (2014), no.~5, 1179--1201.

\bibitem{bdps}
O.J. Boxma, I.~David, D.~Perry, and W.~Stadje, \emph{A new look at organ
  transplantation models and double matching queues}, Probability in the
  Engineering and Informational Sciences \textbf{25} (2011), 135--155.

\bibitem{conolly2002double}
B.~W. Conolly, P.~R. Parthasarathy, and N.~Selvaraju, \emph{Double-ended queues
  with impatience}, Computers \& Operations Research \textbf{29} (2002),
  no.~14, 2053--2072.

\bibitem{DaiHe}
J.~G. Dai and S.~He, \emph{Customer abandonment in many-server queues}, Math.
  Oper. Res. \textbf{35} (2010), no.~2, 347--362.

\bibitem{Dai10}
J.~G. Dai, S.~He, and T.~Tezcan, \emph{Many-server diffusion limits for
  {$G/Ph/n+GI$} queues}, Ann. Appl. Probab. \textbf{20} (2010), no.~5,
  1854--1890.

\bibitem{daiwill96}
J.~G. Dai and R.~J. Williams, \emph{Existence and uniqueness of semimartingale
  reflecting brownian motions in convex polyhedrons}, Theory of Probability \&
  Its Applications \textbf{40} (1996), no.~1, 1--40.

\bibitem{daihe2010}
Jim Dai and Shuangchi He, \emph{Customer abandonment in many-server queues},
  Mathematics of Operations Research \textbf{35} (2010), no.~2, 347--362.

\bibitem{talayasymptotic}
I.~T. Degirmenci, \emph{Asymptotic analysis and performance-based design of
  large scale service and inventory systems}, Ph.D. dissertation, Department of
  Business Administration, Duke University (2010).

\bibitem{dobbie61}
J.~M. Dobbie, \emph{A doubled-ended queuing problem of {K}end}, Operations
  Research \textbf{9} (1961), no.~5, 755--757.

\bibitem{ek86}
Stewart~N. Ethier and Thomas~G. Kurtz, \emph{Markov processes :
  characterization and convergence}, J. Wiley \& Sons, 1986.

\bibitem{Garnet02}
O.~Garnet, A.~Mandelbaum, and M.~Reiman, \emph{Designing a call center with
  impatient customers}, Manufacturing \& Service Operations Management
  \textbf{4} (2002), no.~3, 208--227.

\bibitem{Giveen63}
S.~M. Giveen, \emph{A taxicab problem with time-dependent arrival rates}, SIAM
  Reviw \textbf{5} (1963), no.~2, 119--127.

\bibitem{iw71}
Donald~L. Iglehart and Ward Whitt, \emph{The equivalence of functional central
  limit theorems for counting processes and associated partial sums}, The
  Annals of Mathematical Statistics \textbf{42} (1971), no.~4, 1372--1378.

\bibitem{js03}
J.~Jacod and A.N. Shiryaev, \emph{Limit theorem for stochastic processes}, 2
  ed., Springer-Verlag, Berlin, 2003.

\bibitem{kang10}
Weining Kang and Kavita Ramanan, \emph{Fluid limits of many-server queues with
  reneging}, Ann. Appl. Prob. \textbf{20} (2010), no.~6, 2204--2260.

\bibitem{kashyap1966double}
B.R.K. Kashyap, \emph{The double-ended queue with bulk service and limited
  waiting space}, Operations Research \textbf{14} (1966), no.~5, 822--834.

\bibitem{kendall51}
D.~G. Kendall, \emph{Some problems in the theory of queues}, Journal of the
  Royal Statistical Society. Series B \textbf{13} (1951), no.~2, 151--185.

\bibitem{kim2010simulation}
W.~K. Kim, K.~P. Yoon, G.~Mendoza, and M.~Sedaghat, \emph{Simulation model for
  extended double-ended queueing}, Computers \& Industrial Engineering
  \textbf{59} (2010), no.~2, 209--219.

\bibitem{Lee11}
C.~Lee and A.~Weerasinghe, \emph{Convergence of a queueing system in heavy
  traffic with general patience-time distributions}, Stochastic Processes and
  their Applications \textbf{121} (2011), no.~11, 2507--2552.

\bibitem{liu2015}
Xin Liu, Qi~Gong, and Vidyadhar~G. Kulkarni, \emph{Diffusion models for
  double-ended queues with renewal arrival processes}, Stoch. Syst. \textbf{5}
  (2015), no.~1, 1--61.

\bibitem{Man12}
A.~Mandelbaum and P.~Momcilovic, \emph{Queues with many servers and impatient
  customers}, Mathematics of Operations Research \textbf{37} (2012), no.~1,
  41--65.

\bibitem{pang2007}
Guodong Pang, Rishi Talreja, and Ward Whitt, \emph{Martingale proofs of
  many-server heavy-traffic limits for markovian queues}, Probab. Surveys
  \textbf{4} (2007), 193--267.

\bibitem{perry1999perishable}
D.~Perry and W.~Stadje, \emph{Perishable inventory systems with impatient
  demands}, Mathematical methods of operations research \textbf{50} (1999),
  no.~1, 77--90.

\bibitem{prabhakar2000synchronization}
B.~Prabhakar, N.~Bambos, and T.~S. Mountford, \emph{The synchronization of
  {P}oisson processes and queueing networks with service and synchronization
  nodes}, Advances in Applied Probability \textbf{32} (2000), no.~3, 824--843.

\bibitem{ReedTezcan12}
J.~Reed and T.~Tezcan, \emph{Hazard rate scaling of the abandonment
  distribution for the {$GI/M/n + GI$} queue in heavy traffic}, Oper. Res.
  \textbf{60} (2012), no.~4, 981--995.

\bibitem{Reed08}
J.~E. Reed and Amy~R. Ward, \emph{Approximating the gi/gi/1+gi queue with a
  nonlinear drift diffusion: Hazard rate scaling in heavy traffic}, Math. Oper.
  Res. \textbf{33} (2008), no.~3, 606--644.

\bibitem{WardGlynn03}
A.~R. Ward and P.~W. Glynn, \emph{A diffusion approximation for a {M}arkovian
  queue with reneging}, Queueing Syst. Theory Appl. \textbf{43} (2003),
  no.~1/2, 103--128.

\bibitem{WardGlynn05}
\bysame, \emph{A diffusion approximation for a {$GI/GI/1$} queue with balking
  or reneging}, Queueing Syst. Theory Appl. \textbf{50} (2005), no.~4,
  371--400.

\bibitem{Man05}
S.~Zeltyn and A.~Mandelbaum, \emph{Call centers with impatient customers:
  Many-server asymptotics of the {$M/M/n + G$} queue}, Queueing Systems
  \textbf{51} (2005), no.~3-4, 361--402.

\bibitem{zenios}
S.~A. Zenios, \emph{Modeling the transplant waiting list: {A} queueing model
  with reneging}, Queueing Systems \textbf{31} (1999), 239--251.

\end{thebibliography}

\skp

{\sc

\bigskip\noi
Xin Liu\\
Department of Mathematical Sciences\\
    Clemson University\\
Clemson, SC 29634, USA\\
email:    xliu9@clemson.edu.

}

\end{document}